\documentclass[a4paper]{article}

\usepackage[pages=all, color=black, position={current page.south}, placement=bottom, scale=1, opacity=1, vshift=5mm]{background}

\usepackage[margin=1.2in]{geometry} 

\usepackage{amsmath}
\usepackage{amsthm}
\usepackage{amssymb}
\usepackage[title]{appendix}
\usepackage[utf8]{inputenc}
\usepackage{hyperref}

\usepackage{graphicx}
\usepackage{verbatim,xcolor}
\usepackage{bm}
\usepackage{multirow}
\usepackage{algorithmicx}
\usepackage[ruled]{algorithm}
\usepackage{algpseudocode}
\usepackage[symbol]{footmisc}
\usepackage{float, subcaption}
\usepackage{todonotes}

\usepackage[sort&compress,numbers,square]{natbib}

\theoremstyle{plain}
\newtheorem{theorem}{Theorem}[section]

\newtheorem{lemma}[theorem]{Lemma}

\theoremstyle{definition}
\newtheorem{definition}[theorem]{Definition}

\newtheorem{remark}[theorem]{Remark}

\usepackage{graphicx, color}
\graphicspath{{fig/}}
\usepackage{algorithm, algpseudocode}
\usepackage{mathrsfs}

\usepackage{lipsum}
\numberwithin{equation}{section}

\newcommand{\bx}{{\bf x}}
\newcommand{\bn}{{\bf n}}

\newcommand{\jump}[1]{[ #1 ]}
\newcommand{\avg}[1]{\{ #1\}}

\newcommand{\Th}{\mathcal{T}_h}

\newcommand{\Eh}{\mathcal{E}_h}
\newcommand{\Eho}{\mathcal{E}_h^o}
\newcommand{\Ehb}{\mathcal{E}_h^b}
\newcommand{\bne}{\bn_e}

\newcommand{\norm}[1]{\lVert #1\rVert}

\newcommand{\hnorm}[1]{\lVert #1\rVert_{h}}
\newcommand{\snorm}[1]{|#1|}
\newcommand{\wnorm}[1]{\lVert #1\rVert_{w}}
\newcommand{\trinorm}[1]{{\vert\kern-0.25ex\vert\kern-0.25ex\vert #1 \vert\kern-0.25ex\vert\kern-0.25ex\vert}}

\usepackage{graphicx}
\usepackage{amsmath}
\usepackage{amsfonts}
\usepackage{amsthm}
\usepackage{tikz}
\usepackage{graphicx}
\usepackage{stmaryrd}
\usepackage{enumerate}
\usepackage{mathrsfs}
\usepackage[euler]{textgreek}
\usepackage{listings}
\usepackage{color} 
\usepackage[framed,numbered,autolinebreaks,useliterate]{mcode}
\usepackage[labelsep=period]{caption}
 \usepackage{tikz}

\newcommand{\cE}{\mathcal{E}}
\newcommand{\cT}{\mathcal{T}}

\newcommand{\iO}{\int_{\Omega}}

\newcommand{\argmin}{\mathop{\rm argmin}}

\title{
A $C^0$ weak Galerkin method with preconditioning for constrained optimal control problems with general tracking\thanks{Submitted to a journal in 2025.}
}
\author{SeongHee Jeong,\thanks{Department of Mathematics, Florida State University, Tallahassee, FL 32306 (\texttt{sjeong@fsu.edu})}
\and Seulip Lee,\thanks{Department of Mathematics, Tufts University, Medford, MA 02155 (\texttt{seulip.lee@tufts.edu})}
\and Kening Wang,\thanks{Department of Mathematics and Statistics, University of North Florida, Jacksonville, FL 32224 (\texttt{kening.wang@unf.edu})}
}

\date{
}

\begin{document}
	\maketitle
	
	\begin{abstract}
        This paper presents a $C^0$ weak Galerkin ($C^0$-WG) method combined with an additive Schwarz preconditioner for solving optimal control problems (OCPs) governed by partial differential equations with general tracking cost functionals and pointwise state constraints. These problems pose significant analytical and numerical challenges due to the presence of fourth-order variational inequalities and the reduced regularity of solutions.  
        Our first contribution is the design of a $C^0$-WG method based on globally continuous quadratic Lagrange elements, enabling efficient elementwise stiffness matrix assembly and parameter-free implementation while maintaining accuracy, as supported by a rigorous error analysis. As a second contribution, we develop an additive Schwarz preconditioner tailored to the $C^0$-WG method to improve solver performance for the resulting ill-conditioned linear systems. Numerical experiments confirm the effectiveness and robustness of the proposed method and preconditioner for both biharmonic and optimal control problems.


	\end{abstract}


\section{Introduction}
\label{sec:intro}

Optimal control problems (OCPs) governed by partial differential equations (PDEs) arise naturally in dynamical systems where the goal is to control spatial-temporal processes described by physical laws. 
These problems involve determining control inputs that optimize a performance criterion, such as minimizing energy usage or achieving a desired state distribution, while satisfying the underlying PDE constraints (see e.g.,~\cite{troltzsch2010optimal, meyer2006optimal, hintermuller2010pde}). 
For instance, in heat conduction, the control may represent boundary heating to achieve uniform temperature distribution; in diffusion processes, it may involve regulating sources to maintain concentration levels; and in incompressible fluid flow, optimal control can be used to minimize drag or match velocity profiles by manipulating boundary inflow or body forces. 

Among such OCPs, those involving \textit{a general tracking cost functional} and \textit{pointwise state constraints} are widely used in real-world applications, but often present significant analytical and numerical challenges.
General tracking refers to the objective of matching the system state to a desired target over various spatial features, such as 
points~\cite{allendes2022error, brett2016optimal, allendes2017adaptive}, curves, or subregions in the domain~\cite{brenner2024c0, jeong2025optimal}.
Tracking over specific features inherently reduces the regularity of the optimal state and the associated Lagrange multipliers, potentially resulting in measure-valued multipliers. Moreover, pointwise state constraints, which impose lower and upper bounds on the state variable almost everywhere in the domain, introduce additional complexity.
The state constraints also degrade the regularity of the optimal solution, making it more challenging to enforce constraints in a numerically stable way and to achieve optimal convergence orders.
Therefore, addressing these difficulties remains a central challenge in the analysis and simulation of the constrained OCPs.

To this end, recent studies~\cite{brenner2024c0, jeong2025optimal} propose reformulating the constrained OCPs by eliminating the control variable via the PDE constraint, leading to a reduced formulation with only the state variable.
When pointwise state constraints are applied, the reduced problem becomes a fourth-order variational inequality in the state variable. Once the optimal state is computed, the optimal control can be recovered using the PDE constraint.
For the numerical treatment of such fourth-order formulations, the $C^0$ interior penalty ($C^0$-IP) methods, originally introduced in~\cite{brenner2005c}, utilize globally continuous Lagrange elements with penalty terms to weakly enforce the continuity of first derivatives using a sufficiently large penalty parameter. Compared to conforming $C^1$ finite elements, the $C^0$-IP methods are appealing for fourth-order problems such as biharmonic problems, as they employ standard Lagrange polynomial spaces, minimizing the degrees of freedom. These methods were recently adopted in~\cite{brenner2024c0, jeong2025optimal} to solve the constrained OCPs.
However, the $C^0$-IP method in the symmetric case relies on careful tuning of penalty parameters to ensure stability and convergence, and it includes several trace terms that interfere with elementwise (or parallel) assembly of the stiffness matrix.

In this paper, we propose a weak Galerkin (WG) approach for the constrained OCPs using globally continuous quadratic Lagrange elements, aiming to enhance the efficiency of both numerical discretizations and linear solvers.
WG methods, based on weakly defined differential operators, offer several advantages across a variety of PDEs, including flexibility for complex geometries through the use of polygonal meshes, elementwise local formulations, and the ability to vary polynomial degrees across elements (see, e.g.,~\cite{wang2013weak,li2013weak,mu2013weakInterface,mu2015weak,mu2015weakMaxwell,wang2016weak}).
Especially for biharmonic problems, WG methods based on the weak Laplacian have been developed in~\cite{mu2014c, mu2014weak, ye2020stabilizer, zhu2023stabilizer}. A recent study~\cite{lee2024low} further demonstrated that a WG formulation using the weak gradient can significantly improve the efficiency of stiffness matrix assembly by replacing symmetric interior penalty formulations in the Stokes equations~\cite{yi2022enriched,hu2024pressure}.
Motivated by these advances, we adopt a WG framework to address the constrained OCPs governed by a fourth-order variational inequality and subject to pointwise state constraints, where efficient and robust biharmonic discretizations are essential. 

Our first contribution is the development of a $C^0$ weak Galerkin ($C^0$-WG) method that employs globally continuous quadratic Lagrange elements to simplify stiffness matrix assembly, eliminating the trace terms and the penalty parameter required in the symmetric $C^0$-IP method. The discrete weak Laplacian is locally computed as a piecewise constant function in each element, using function values at the barycenters of interfaces shared by two neighboring elements, along with basic geometric data such as lengths of edges and areas of elements in the two-dimensional setting (details will be provided in Section~\ref{sec: c0WG_simply_supported}).
In the $C^0$-WG method, the bilinear form is assembled from the $L^2$ inner product of the weak Laplacians and a parameter-free penalty (or stabilization) term.
As a result, the proposed method offers efficient matrix assembly and stable performance independent of tuning parameters.
Furthermore, we present a rigorous error analysis within the WG framework, establishing the well-posedness and accuracy of the method for the constrained OCPs by drawing a connection to the existing $C^0$-IP method~\cite{brenner2024c0}.

As a second contribution, we develop an additive Schwarz preconditioner tailored to the $C^0$-WG method to address the computational challenges posed by ill-conditioned systems arising from fourth-order PDEs. In particular, when the cost functional involves general tracking terms, the resulting system matrix often exhibits a large condition number, which can severely degrade solver performance. To mitigate this issue, we adopt a one-level additive Schwarz preconditioner within the framework of~\cite{brenner2022additive}. This domain decomposition technique partitions the global problem into overlapping subdomain problems, each solved independently, and combines their solutions additively to form a global approximation. This approach is particularly effective for parallel computing and has been successfully applied in large-scale simulations~\cite{brenner2020Balancing, park2024Additive}. For the proposed WG formulation, the preconditioner significantly improves the boundedness of the condition number, enabling efficient and scalable iterative solvers for the constrained OCPs.

The $C^0$-WG method and preconditioner are developed for both two- and three-dimensional settings. For the constrained OCPs, the discussion focuses on two dimensions due to the complexity introduced by the general tracking terms; however, the approach can be extended to three dimensions without significant difficulty. Accordingly, the numerical experiments are also presented in two-dimensional settings.

The remainder of this paper is organized as follows.
Section~\ref{sec:prelim} introduces key notation and definitions, presents the constrained OCPs, and reviews existing approaches.
In Section~\ref{sec: c0WG_simply_supported}, we recall the definition of the weak Laplacian and introduce a WG discretization for the constrained OCPs.
Section~\ref{sec: convergence_analysis} provides a convergence analysis for the $C^0$-WG method applied to the OCPs.
An additive Schwarz preconditioner for the resulting discrete system is discussed in Section~\ref{sec: preconditioning}.
Finally, Section~\ref{sec: numerical_experiment} presents numerical experiments for the proposed numerical method applied to both biharmonic problems and the constrained OCPs, including a comparison of condition numbers with and without the proposed preconditioner. We summarize our contributions and discuss related research directions in Section~\ref{sec:conclusion}.


\section{Preliminaries}
\label{sec:prelim}

In this section, we introduce key notation and definitions, present the optimal control problems with general tracking and state constraints, and review existing numerical approaches.


\subsection{Notation and Definitions}

We introduce the notation and definitions used throughout this paper.
Let $\mathcal{D} \subset \mathbb{R}^d$ be a bounded Lipschitz domain, where $d=2$ or 3. For any real number $s \geq 0$, we denote the Sobolev space on $\mathcal{D}$ by $H^s(\mathcal{D})$, with associated norm $\|\cdot\|_{s,\mathcal{D}}$ and seminorm $|\cdot|_{s,\mathcal{D}}$. In particular, $H^0(\mathcal{D}) = L^2(\mathcal{D})$, and the $L^2$ inner product is denoted by $(\cdot,\cdot)_{\mathcal{D}}$. When $\mathcal{D} = \Omega$, where $\Omega$ is the domain associated with the optimal control problems (see Section~\ref{subsec: OCP}), the subscript $\mathcal{D}$ will be omitted for brevity. This notation naturally extends to vector- and tensor-valued Sobolev spaces.
We denote by $H_0^1(\mathcal{D})$ the space of functions in $H^1(\mathcal{D})$ with vanishing trace on $\partial\mathcal{D}$, and by $H^2_0(\mathcal{D})$ the space of functions in $H^2(\mathcal{D})$ with vanishing trace and vanishing normal flux on $\partial \mathcal{D}$. 
The space of polynomials of degree less than or equal to $k$ on $\mathcal{D}$ is denoted by $P_k(\mathcal{D})$.
We also define the Hilbert space
\begin{equation*}
    H(\mathrm{div}, \mathcal{D}) := \left\{ \mathbf{v} \in [L^2(\mathcal{D})]^2 \,:\, \mathrm{div}\,\mathbf{v} \in L^2(\mathcal{D}) \right\}
\end{equation*}
equipped with the norm
\begin{equation*}
    \|\mathbf{v}\|_{H(\mathrm{div},\mathcal{D})}^2 := \|\mathbf{v}\|_{0,\mathcal{D}}^2 + \|\mathrm{div}\,\mathbf{v}\|_{0,\mathcal{D}}^2.
\end{equation*}

For the discrete setting, let $\mathcal{T}_h$ be a shape-regular triangulation of $\Omega$, where each element $T \in \mathcal{T}_h$ is either a triangle or a tetrahedron. Let $\mathcal{E}_h$ denote the set of all edges or faces, which is partitioned into interior edges or faces $\mathcal{E}_h^o$ and boundary edges or faces $\mathcal{E}_h^b$.
For any element $T \in \mathcal{T}_h$, we denote its diameter by $h_T$ and its outward unit normal vector on $\partial T$ by $\mathbf{n}_T$. For each interior interface $e \in \mathcal{E}_h^o$ shared by neighboring elements $T^+$ and $T^-$, we define $\mathbf{n}_e$ to be the unit normal vector pointing from $T^-$ to $T^+$. For a boundary edge or face $e \in \mathcal{E}_h^b$, $\mathbf{n}_e$ denotes the outward unit normal vector on $\partial\Omega$.

The broken Sobolev space associated with $\mathcal{T}_h$ is defined as
\begin{equation*}
    H^s(\mathcal{T}_h) := \left\{ v \in L^2(\Omega) \,:\, v|_T \in H^s(T),\ \forall T \in \mathcal{T}_h \right\},
\end{equation*}
with the norm
\begin{equation*}
    \|v\|_{s,\mathcal{T}_h} := \left( \sum_{T \in \mathcal{T}_h} \|v\|_{s,T}^2 \right)^{1/2}.
\end{equation*}
For $s = 0$, the $L^2$ inner product over $\mathcal{T}_h$ is denoted by $(\cdot,\cdot)_{\mathcal{T}_h}$. Similarly, the $L^2$ inner product and norm on $\mathcal{E}_h$ are denoted by $\langle\cdot,\cdot\rangle_{\mathcal{E}_h}$ and
\begin{equation*}
    \|v\|_{0,\mathcal{E}_h} := \left( \sum_{e \in \mathcal{E}_h} \|v\|_{0,e}^2 \right)^{1/2}.
\end{equation*}
The piecewise polynomial space over $\mathcal{T}_h$ is defined as
\begin{equation*}
    P_k(\mathcal{T}_h) := \left\{ v \in L^2(\Omega) \,:\, v|_T \in P_k(T),\ \forall T \in \mathcal{T}_h \right\}.
\end{equation*}

For any function $v$, the jump and average across $e \in \mathcal{E}_h$ are defined as
\begin{equation*}
    \jump{v} = \begin{cases}
        v^+ - v^- & \text{on } e \in \mathcal{E}_h^o, \\
        v & \text{on } e \in \mathcal{E}_h^b,
    \end{cases}
    \qquad
    \avg{v} = \begin{cases}
        (v^+ + v^-)/2 & \text{on } e \in \mathcal{E}_h^o, \\
        v & \text{on } e \in \mathcal{E}_h^b,
    \end{cases}
\end{equation*}
where $v^\pm$ denotes the trace of $v|_{T^\pm}$ on $e = \partial T^+ \cap \partial T^-$.
These definitions extend naturally to vector- and tensor-valued functions.

We also recall the following identity and trace inequality used in this paper. For any vector-valued function $\mathbf{v}$ and scalar function $q$, we have
\begin{equation}\label{eqn: jump-avg}
    \sum_{T \in \mathcal{T}_h} \langle \mathbf{v} \cdot \mathbf{n}_T, q \rangle_{\partial T}
    = \langle \jump{\mathbf{v}} \cdot \mathbf{n}_e, \avg{q} \rangle_{\mathcal{E}_h} + \langle \avg{\mathbf{v}} \cdot \mathbf{n}_e, \jump{q} \rangle_{\mathcal{E}_h^o}.
\end{equation}
Moreover, for any $v \in H^1(T)$ and $e \subset \partial T$, the following trace inequality holds:
\begin{equation}\label{eqn: trace}
    \|v\|_{0,e}^2 \leq C\left( h_T^{-1} \|v\|_{0,T}^2 + h_T \|\nabla v\|_{0,T}^2 \right),
\end{equation}
where $C > 0$ is a constant depending only on the shape regularity of the mesh.


\subsection{Optimal Control Problems}\label{subsec: OCP}

Let $\Omega\subset\mathbb{R}^2$ be a convex, bounded, and connected Lipschitz domain with boundary $\partial \Omega$. 
The optimal control problems (OCPs) considered in this work aim to find
\begin{equation}\label{ocp}
\left(\bar{y},\bar{u}\right)=\argmin_{(y,u)\in \mathbb{U}} \frac1{2}\left[\int_\Omega |y-y_d|^2 d\chi 
    + \beta\int_\Omega |u|^2 d\bx\right],
\end{equation}
where $\mathbb{U}\subset H^1_0(\Omega)\times L^2(\Omega)$, $y_d$ is a desired state (to be defined in~\eqref{eqn: defi_y_bar}), and $\beta>0$. The optimization is subject to the elliptic PDE:
\begin{subequations}
\begin{alignat}{2}
-\Delta y = u\quad&\text{in }\Omega, \label{state_control}\\
y=0\quad&\text{on }\partial\Omega,\label{state_bd}
\end{alignat}
\end{subequations}
and the pointwise state constraints:
\begin{equation}
    \psi_-(\bx)\leq y(\bx)\leq \psi_+(\bx)
\quad\text{a.e. }\;\bx=(x,y)\in \Omega.
\label{eqn:point_constraints}
\end{equation}
In \eqref{ocp}, we recall the definition of the Radon measure $\chi$ on $\bar\Omega$ in~\cite{brenner2024c0, jeong2025optimal}:
\begin{equation}\label{radon}
    \iO f\,d\chi = \sum_{j=1}^J f(\mathscr{P}_j)w^j_\mathscr{P}
    +\sum_{l=1}^L\int_{\mathscr{C}_l}f w^l_\mathscr{C}\,ds
    +\sum_{m=1}^M \int_{\mathscr{E}_m} fw^m_\mathscr{E}\,d\bx,
\end{equation}
where:

\begin{itemize}
    \item $\mathscr{P}=\{\mathscr{P}_1,\ldots,\mathscr{P}_J\}$ is a finite set of points in $\Omega$,
    \item $\mathscr{C}=\{\mathscr{C}_1,\ldots,\mathscr{C}_L\}$ 
    is a collection of curves where $\mathscr{C}_l\subset\Omega$,
    \item $\mathscr{E} = \{  \mathscr{E}_1, \ldots, \mathscr{E}_M \}$ is a collection of subdomains $\mathscr{E}_m \subset \Omega$.
\end{itemize}
The weight functions $w_\mathscr{P}^j$, $w_\mathscr{C}^l$ and $w^m_\mathscr{E}$ are bounded, nonnegative Borel-measurable functions defined on $\mathscr{P}$, $\mathscr{C}$ and $\mathscr{E}$, respectively. 
The desired state $y_d$ is a target function for the state function $y$, defined by
\begin{equation}y_d := \left\{
    \begin{array}{cll}
         y_\mathscr{P}&\text{on }\mathscr{P},  \\
        y_\mathscr{C}&\text{on }\mathscr{C}\setminus\mathscr{P},\\
        y_\mathscr{E}&\text{on }\mathscr{E}\setminus
         (\mathscr{C}\cup\mathscr{P}),
    \end{array}\right.\label{eqn: defi_y_bar}
\end{equation}
such that 
\[\|y_d\|^2_{L^2(\Omega;\chi)}:=\iO |y_d|^2 d\chi<\infty.\]
We also assume that the functions $\psi_\pm$ in \eqref{eqn:point_constraints} satisfy the following conditions~\cite{brenner2024c0}:
\begin{subequations}
\begin{alignat}{2}
&\psi_\pm\in W^{3,q}(\Omega)&&\text{\quad for }q>2,\label{state1}\\
&\psi_-<\psi_+&&\text{\quad on }\bar\Omega,\label{state2}\\
&\psi_-<0<\psi_+&&\text{\quad on }\partial\Omega.\label{state3}
\end{alignat}
\end{subequations}

In this work, instead of solving simultaneously for both the state $y$ and the control $u$, we reformulate the OCPs~\eqref{ocp} by substituting the elliptic PDE~\eqref{state_control} directly into the cost functional.
\begin{algorithm}[H]
\caption*{\textbf{Reduced Optimal Control Problems}}\label{alg: reduced_ocp}
Find $\bar{y}\in U$ such that
\begin{equation}\label{reduced_cost}
   \bar{y} = \argmin_{y\in U} \frac1{2}\left[\iO |y-y_d|^2 d\chi 
    + \beta\iO |\Delta y|^2 d\bx\right],
\end{equation}
where the admissible set is given by
\begin{equation}\label{admissible}
    U = \{ y\in H^2(\Omega)\cap H^1_0(\Omega)\,:\,
    \psi_-\leq y\leq \psi_+
    \text{ in }\Omega\}.
\end{equation}
Note that, according to (\ref{state2})-(\ref{state3}),
the admissible set $U$ is nonempty.
\end{algorithm}

\begin{remark}
    As a result of the convexity of $\Omega$, the constraint~\eqref{state_control} implies
    that $y\in H^2(\Omega)$ for any 
    $(y,u)\in \mathbb{U}$, by elliptic regularity.
    Therefore, the cost functional in~\eqref{ocp}
    is well-defined by the Sobolev embedding theorem~\cite{adams2003sobolev} ($H^2(\Omega)\subset C(\bar\Omega)$). 
\end{remark}

The reduced OCPs~\eqref{reduced_cost} admits a unique solution $\bar{y}\in U$, which is characterized by the following fourth-order variational inequality~\cite{kinderlehrer2000introduction}:
\begin{equation}\label{vi}
\beta\iO (\Delta\bar{y})(\Delta(y-\bar{y}))\,d\bx
+\iO(\bar{y}-y_d)(y-\bar{y})\,d\chi
\geq 0,\quad \forall y\in U.
\end{equation}
This variational inequality is equivalent to the following generalized Karush-Kuhn-Tucker (KKT) conditions~\cite{ito2008lagrange,rudin1987real}:
    \begin{equation}\label{kkt}
    \beta\iO(\Delta\bar{y})(\Delta z)\,d\bx
    +\iO(\bar{y}-y_d)z\,d\chi=
    \iO z\,d\xi,
    \quad\forall z\in H^2(\Omega)\cap H^1_0(\Omega),
    \end{equation}
    where $\xi$ is a bounded regular Borel measure satisfying the conditions
\begin{subequations}\label{lagrange}
\begin{alignat}{2}
        \xi \geq 0 &\quad\text{if } \bar{y}=\psi_-,\label{lagrange1}\\
        \xi \leq 0 &\quad\text{if } \bar{y}=\psi_+,\label{lagrange2}\\
        \xi =0 &\quad\text{otherwise}.\label{lagrange3}
\end{alignat}
\end{subequations}
We refer the reader to~\cite{jeong2023} for a detailed derivation of the KKT conditions.

We conclude this section by presenting the regularity results for the optimal state $\bar{y}$ of the reduced problem~\eqref{reduced_cost}.
It is known~\cite{agmon1959estimates, hormander1985analysis, brenner2024c0} that the solution satisfies the local regularity:
$$\bar{y}\in W^{3,s}_{\text{loc}}(\Omega),\quad \forall s\in (1,2).$$
Furthermore, by global regularity theory~\cite{grisvard2011elliptic, dauge1988elliptic, brenner2024c0}, we have  
\begin{equation}
    \bar{y}\in H^{2+\alpha}(\Omega),
\end{equation}
for some $\alpha\in(0,1)$, where $\alpha$ depends on the geometric regularity of $\Omega$. In the case where $\Omega$ is a rectangle, it is known
that $\alpha=1-\epsilon$ for any $\epsilon>0$; see~\cite{dauge1988elliptic, casas19852} for more details.


\subsection{Existing Approaches to the Optimal Control Problems}\label{subsec: C0IPM}

We begin by introducing the bilinear form
(see, e.g.,~\cite{grisvard2011elliptic}):
\begin{equation}\label{biharmonicbi}
    a(v,w) =\iO (\Delta v)(\Delta w)\,d\bx= \iO D^2v:D^2w\,d\bx,
    \quad \forall v,w\in H^2(\Omega)\cap H^1_0(\Omega),
\end{equation}
where $D^2v$ denotes the Hessian matrix of $v$, and $D^2v : D^2w$ represents the Frobenius inner product of $D^2v$ and $D^2w$.

Let $V_h\subset H^1_0(\Omega)$ be the Lagrange finite element space of degree $k\geq 2$ associated with
$\cT_h$, consisting of globally continuous, piecewise polynomials from $P_k(\Th)$.
The discrete bilinear form in the $C^0$ interior penalty ($C^0$-IP) method \cite{brenner2005c} is defined for all $v,w\in V_h$ as
\begin{align}
    a_h(v,w) &= (D^2 v,D^2 w)_{\mathcal{T}_h} + \langle \avg{\nabla(\nabla v\cdot \bn_e)}\cdot\bn_e,\jump{\nabla w}\cdot\bn_e\rangle_{\Eho}\nonumber\\
    &\qquad\qquad+ \langle \avg{ \nabla(\nabla w\cdot\bn_e)}\cdot\bn_e,\jump{\nabla v}\cdot\bn_e\rangle_{\Eho} +  \rho\langle h_e^{-1}\jump{\nabla v}\cdot\bn_e,\jump{\nabla w}\cdot\bn_e\rangle_{\Eho}.\label{eqn: C0IP_bilinear_form}
\end{align} 
Here, $\rho > 0$ is a penalty parameter, and $h_e=|e|^{1/(d-1)}$, where $|e|$ denotes the measure of the interfaces $e \in \Eho$.

In the discrete bilinear form, although the finite element functions in $V_h$ are globally continuous, their normal derivatives are discontinuous across interfaces $e\in\Eho$ because the functions do not belong to $H^2(\Omega)$. This discontinuity in the normal derivatives gives rise to the jump and average terms across interfaces in the discrete bilinear form.  
Moreover, thanks to the symmetric structure of the discrete bilinear form and the inclusion of a penalty term, the bilinear form $a_h(\cdot,\cdot)$ becomes symmetric positive-definite (SPD) with a sufficiently large penalty parameter $\rho$, ensuring that the discrete problem preserves the SPD property of the continuous bilinear form $a(\cdot,\cdot)$ in \eqref{biharmonicbi}.

Using the $C^0$-IP method, we discretize the reduced OCPs~\eqref{reduced_cost} (see~\cite{brenner2024c0, jeong2025optimal} for futher details), formulated as
\begin{equation}
    \bar{y}_h = \argmin_{y_h\in U_h}\frac1{2}\left[
    \beta a_h(y_h,y_h) + \|y_h-y_d\|^2_{L^2(\Omega;\chi)}\right],
\end{equation}
where $a_h(\cdot,\cdot)$ is defined in~\eqref{eqn: C0IP_bilinear_form}, and $U_h$ is a discrete admissible set corresponding to $U$ in~\eqref{admissible} (details will be provided in Section~\ref{subsec: discretize_OCPs}).

\begin{remark}
    The $C^0$-IP method offers a significant advantage over $C^1$-conforming and other higher-order finite element methods by requiring a minimal number of degrees of freedom.  
Whereas $C^1$-conforming elements necessitate complex shape functions and a substantial increase in degrees of freedom to achieve higher continuity, the $C^0$-IP method overcomes these challenges by employing standard Lagrange elements and enforcing higher continuity weakly through interior penalty terms.  
As a result, it enables a much simpler implementation and more efficient use of computational resources, particularly for fourth-order problems.
\end{remark}

\begin{remark}
    The $C^0$-IP method~\cite{brenner2005c} for the biharmonic problem in $H^2(\Omega)\cap H_0^1(\Omega)$,
\begin{subequations}\label{sys:governing}
\begin{alignat}{2}
\Delta^2 u &= f \quad\text{in }\Omega, \label{mp1}\\
u=\Delta u & = 0\quad\text{on }\partial\Omega,\label{mp3}
\end{alignat}
\end{subequations}
seeks $u_h \in V_h$ such that
    \begin{equation}\label{eqn: C0IP_biharmonic}
    a_h(u_h,v_h)=(f,v_h),\quad \forall v_h\in V_h.
\end{equation}
With the mesh-dependent norm
\begin{equation}\label{eqn: defi_h_norm}
    \hnorm{v}^2 = (D^2v,D^2v)_{\Th} + \langle h_e^{-1}\jump{\nabla v}\cdot\bne,\jump{\nabla v}\cdot\bne\rangle_{\Eho},
\end{equation}
there exist positive constants $\eta_*$ and $\eta^*$ such that
\begin{align}
    &a_h(v,v)\geq \eta_*\hnorm{v}^2,\label{eqn: hcoera}\\
    &a_h(v,w)\leq \eta^*\hnorm{v}\hnorm{w}.\label{eqn: hcontia}
\end{align}
Therefore, the discrete problem~\eqref{eqn: C0IP_biharmonic} is well-posed; see~\cite{brenner2005c} for further details.
\end{remark}


\section{A $C^0$ Weak Galerkin Method for the Optimal Control Problems}\label{sec: c0WG_simply_supported}

In this section, we propose a $C^0$ weak Galerkin ($C^0$-WG) finite element method for the optimal control problems described in \eqref{reduced_cost}, using globally continuous quadratic Lagrange elements together with the weak Laplacian~\cite{mu2014weak, mu2014c}.
In contrast to the WG methods in~\cite{mu2014weak, mu2014c}, which are formulated for the space $H_0^2(\Omega)$, our method addresses biharmonic and optimal control problems in the space $H^2(\Omega)\cap H_0^1(\Omega)$.
When applied to biharmonic problems using the same quadratic Lagrange elements, the proposed $C^0$-WG method results in a bilinear form that is simpler than that of the $C^0$-IP method~\eqref{eqn: C0IP_bilinear_form}.

\subsection{Weak Laplacian}

We introduce the space of weak functions~\cite{mu2014weak,mu2014c} on each element $T\in\mathcal{T}_h$:
\begin{equation}
    \mathcal{W}(T)= \left\{
    \nu=\left\{\nu_0, \nu_b, \bm{\nu}_g\right\}\;: \;
    \nu_0\in L^2(T),
    \ \nu_b\in H^{\frac1{2}}(\partial T),
    \ \bm{\nu}_g\cdot \bn_T\in H^{-\frac1{2}}(\partial T)
    \right\}
\end{equation}
Each component of $\nu$ corresponds to a term arising from the integration by parts formula for the Laplacian:
\begin{eqnarray*}
    (\Delta v,\varphi)_T
    = (v,\Delta \varphi)_T
    -\langle v, \nabla\varphi\cdot\bn_T\rangle_{\partial T}
    +\langle \nabla v\cdot\bn_T,\varphi\rangle_{\partial T}.
\end{eqnarray*}
More precisely, the correspondences are given by
\begin{equation*}
    v|_T\ \leftrightarrow\ \nu_0,\quad v|_{\partial T}\ \leftrightarrow\  \nu_b,\quad\text{and}\quad\nabla v|_{\partial T}\ \leftrightarrow\ \bm{\nu}_g.
\end{equation*}
We also introduce the test function space
\begin{equation*}
    \mathcal{G}_2(T) = \left\{
    \varphi\; : \; \varphi \in H^1(T),
    \ \Delta\varphi\in L^2(T)\right\},
\end{equation*}
highlighting that
$\displaystyle\nabla\varphi\in H(\text{div},T)$, so
$\displaystyle\nabla\varphi\cdot\bn_T\in H^{-\frac1{2}}(\partial T)$ for any $\varphi\in \mathcal{G}_2(T)$.

\begin{definition}[{Weak Laplacian~\cite{mu2014weak,mu2014c}}]
    The weak Laplacian of $\nu=\{\nu_0,\nu_b,\bm{\nu_g}\}\in \mathcal{W}(T)$ is defined as a linear functional $\Delta_w\nu$ in the dual space of $\mathcal{G}_2(T)$, whose action is given by
    \begin{equation}
    (\Delta_w \nu,\varphi)_T
    = (\nu_0,\Delta \varphi)_T
    -\langle \nu_b, \nabla\varphi\cdot\bn_T\rangle_{\partial T}
    +\langle \bm{\nu}_g\cdot\bn_T,\varphi\rangle_{\partial T},\quad\forall \varphi\in \mathcal{G}_2(T).
\end{equation}
\end{definition}

In a finite-dimensional setting, we introduce the discrete weak function space~\cite{mu2014weak},
\begin{equation*}
    \mathcal{W}_k(T) = \left\{
    \nu=\left\{\nu_0, \nu_b,\bm{\nu}_g\right\}\;:\ 
    \nu_0\in P_k(T),\ 
    \nu_b\in P_k(e),\ 
    \bm{\nu}_g\in[P_{k-1}(e)]^d,\ e\subset\partial T\right\}
\end{equation*}
for any given integer $k\geq 2$.

\begin{definition}[Discrete weak Laplacian operator~\cite{mu2014weak}]
    For $\nu = \{\nu_0,\nu_b,\bm{\nu}_g\}\in \mathcal{W}_k(T)$, the discrete weak Laplacian operator is defined as the unique polynomial
$\Delta_{w}\nu\in P_r(T)$ for $r\geq 0$ satisfying
\begin{equation}\label{discreteL}
    (\Delta_{w}\nu,p_r)_T
    =(\nu_0,\Delta p_r)_T
    -\langle \nu_b,\nabla p_r\cdot \bn_T\rangle_{\partial T}
    +\langle \bm{\nu}_g\cdot\bn_T, p_r\rangle_{\partial T},
    \quad\forall p_r\in P_r(T).
\end{equation}
\end{definition}

We emphasize that a continuous (or $C^0$) WG method, first introduced in~\cite{mu2014c}, considers discrete weak functions $\nu = \{\nu_0,\nu_0,\bm{\nu}_g\}$, where $\nu_b$ is simply the trace of $\nu_0$ on $\partial T$.
Moreover, a global WG finite element space~\cite{mu2014weak} is defined by patching $\mathcal{W}_k(T)$ over all the elements $T\in\Th$ while enforcing common values on interior interfaces $e\in \Eho$:
\begin{equation*}
    \mathcal{W}_{k}(\Th) = \left\{
    \nu = \{\nu_0, \nu_b,\bm{\nu}_g\}\,:\,
    \nu|_T=\{\nu_0|_T, \nu_b|_{\partial T},\bm{\nu}_g|_{\partial T}\}\in \mathcal{W}_k(T),\ \forall T\in\Th\right\}.
\end{equation*}
The corresponding subspace incorporating homogeneous boundary conditions is defined as
\begin{equation*}
    \mathcal{W}_{k}^0(\Th) = \left\{
    \nu = \{\nu_0, \nu_b,\bm{\nu}_g\}\in \mathcal{W}_k(\Th)\,:\,
     \nu_b|_{e}=0,\ \forall e\subset\partial \Omega\right\}.
\end{equation*}

Therefore, the general WG finite element method~\cite{mu2014weak} for biharmonic problems seeks $\mu=\{\mu_0, \mu_b,\bm{\mu}_g\}\in \mathcal{W}_k^0(\Th)$ such that
\begin{equation}\label{eqn: WG_biharmonic}
    (\Delta_w \mu, \Delta_w \nu)_{\mathcal{T}_h} + s(\mu,\nu) = (f,\nu_0),
    \quad \forall \nu=\{\nu_0, \nu_b,\bm{\nu}_g\}\in \mathcal{W}_k^0(\Th),
\end{equation}
where the stabilizer $s(\cdot,\cdot)$ is defined by
\begin{equation}\label{penalty}
    s(\mu,\nu)
    = \langle h_e^{-1}\jump{\nabla \mu_0-\bm{\mu}_g}\cdot\bne,\jump{\nabla \nu_0-\bm{\nu}_g}\cdot\bne\rangle_{\Eh} + \sum_{T\in\Th} \langle h_T^{-3}(\mu_0-\mu_b),\nu_0-\nu_b\rangle_{\partial T}.
\end{equation}

We observe that, in the continuous setting (i.e., $\nu_b=\nu_0$ on $\partial T$), the second term of the stabilizer vanishes, resulting in a simpler stabilizer.
Furthermore, selecting $r>k$ in \eqref{discreteL} allows the formulation in \eqref{eqn: WG_biharmonic} to eliminate the need for the stabilizer, as demonstrated in stabilizer-free WG methods~\cite{ye2020stabilizer, zhu2023stabilizer}.
However, using such high-order discrete weak Laplacian operators significantly increases the computational complexity, especially for $r>k\geq 2$.


\subsection{A Weak Galerkin Interpretation of the $C^0$ Interior Penalty Method}

As discussed in Section~\ref{subsec: C0IPM}, the $C^0$-IP method~\cite{brenner2005c} uses a finite-dimensional space based on globally continuous, piecewise $P_k$ Lagrange polynomials, which reduces the degrees of freedom for biharmonic problems while utilizing interior penalty methods.
In this context, we interpret the $C^0$-IP finite element space incorporating the discrete weak Laplacian operator \eqref{discreteL} and demonstrate that this approach results in a simplified formulation for biharmonic problems.

We recall that $V_h\subset H_0^1(\Omega)$ is the set of globally continuous, piecewise $P_k$ Lagrange polynomials used in the $C^0$-IP method. While $k \geq 2$ in general, we fix $k = 2$ to minimize the degrees of freedom and to derive an efficient formula for the discrete weak Laplacian.
Then, any function $v_h\in V_h$ can be interpreted as a discrete weak function in $\mathcal{W}_2(\Th)$.
Specifically,
\begin{equation*}
    \nu_0=v_h, \quad \nu_b=v_h,\quad \bm{\nu}_g = \avg{\nabla v_h}\quad\Rightarrow\quad
    \{v_h,v_h,\avg{\nabla v_h}\}\in\mathcal{W}_2(\Th),
\end{equation*}
since $v_h$ is continuous, but its flux is discontinuous.
Moreover, the discrete weak Laplacian operator \eqref{discreteL} for $v_h\in V_h$ is defined as $\Delta_w v_h|_T\in P_r(T)$ for $r\geq 0$, satisfying
\begin{equation*}
     (\Delta_{w}v_h,p_r)_T
    =(v_h,\Delta p_r)_T
    -\langle v_h,\nabla p_r\cdot \bn_T\rangle_{\partial T}
    +\langle \avg{\nabla v_h}\cdot\bn_T, p_r\rangle_{\partial T},
    \quad\forall p_r\in P_r(T).
\end{equation*}
While various options exist for choosing the polynomial degree $r\geq0$, we primarily focus on the lowest-order discrete weak Laplacian, where $\Delta_w v_h|_T\in P_0(T)$, and compare it with $\Delta v_h|_T\in P_0(T)$ in the $C^0$-IP method.
In the case where $p_0\in P_0(T)$, the formula for the discrete weak Laplacian for $v_h\in V_h$ simplifies to
\begin{equation}\label{eqn: weak_defi} 
    (\Delta_w v_h,p_0)_T = \langle \avg{\nabla v_h}\cdot \bn_T,
    p_0\rangle_{\partial T},\quad\forall p_0\in P_0(T).
\end{equation}

\begin{remark}
For any $v_h\in V_h$, the difference between the discrete weak Laplacian and the standard Laplacian is expressed as
    \begin{equation}
    \left(\Delta_w v_h-\Delta v_h,p_0\right)_{\Th}=\langle\jump{\nabla v_h}\cdot \bn_e,\avg{p_0}\rangle_{\Eho},\quad\forall p_0\in P_0(\Th).\label{eqn: relation1}
\end{equation}
This identity follows directly from the definition of the discrete weak Laplacian, integration by parts, and the trace property \eqref{eqn: jump-avg}.
\end{remark}

Since both $\Delta_w v_h$ and $p_0$ are constant in each $T\in\Th$, we set $p_0=1$ and divide by $|T|$ on both sides to derive an explicit expression for the weak Laplacian on $T$,
\begin{equation*}
    \left.\Delta_w v_h\right|_T = \frac{1}{|T|}\int_{\partial T}\avg{ \nabla v_h}\cdot \bn_T\,ds.
\end{equation*}
We note that $\avg{\nabla v_h}\cdot \bn_T$ is a linear function on each edge or face $e\subset \partial T$, which ensures that the one-point quadrature rule yields exact integration.
Consequently, the final expression becomes
\begin{equation}\label{eqn: discreteweakL}
    \left.\Delta_w v_h\right|_{T} = \frac1{|T|}\sum_{e\subset\partial T}
    \frac{|e|}{2}
    \big( \nabla v_h^+(\mathbf{m}_e) + \nabla v_h^-(\mathbf{m}_e)\big)\cdot (\bn_T|_e),
\end{equation}
where $e\in\Eho$ is an interior interface shared by neighboring elements $T^+$ and $T^-$, and $\mathbf{m}_e$ denotes the barycenter of $e$.
When $e\in\Ehb$, the formula~\eqref{eqn: discreteweakL} is adjusted by using $\avg{\nabla v_h}=\nabla v_h$.

Additionally, for the stabilizer $s(\cdot,\cdot)$ in \eqref{penalty}, we observe the following:
\begin{equation*}
    \nu_0-\nu_b=v_h-v_h=0\quad\text{and}\quad\nabla\nu_0-\bm{\nu}_g = \nabla v_h-\avg{\nabla v_h}=\pm\frac{1}{2}\jump{\nabla v_h},
\end{equation*}
where the sign of the jump depends on the orientation of the elements sharing an interface.
On a boundary edge or face $e \in\Ehb$, it is straightforward to verify that $\nu_0 - \nu_b  = v_h - v_h = 0$ and $\nabla \nu_0 - \bm{\nu}_g = \nabla v_h - \avg{\nabla v_h} = \nabla v_h - \nabla v_h = \mathbf{0}$.
Therefore, the stabilizer $s(u_h,v_h)$ for $u_h,v_h\in V_h$ is defined as
\begin{equation}\label{eqn: final_stabz}
    s(u_h,v_h)
    = \frac{1}{4}\langle h_e^{-1}\jump{\nabla u_h}\cdot\bne,\jump{\nabla v_h}\cdot\bne\rangle_{\Eho}.
\end{equation}

\begin{remark}
    In contrast to the $C^0$-IP method, which requires a sufficiently large penalty parameter $\rho$ in the penalty term $\rho\langle h_e^{-1}\jump{\nabla u_h}\cdot\bne,\jump{\nabla v_h}\cdot\bne\rangle_{\Eho}$, our proposed method is parameter-free. 
Moreover, the formula~\eqref{eqn: discreteweakL} enables elementwise assembly of the stiffness matrix without the need to compute certain trace terms in~\eqref{eqn: C0IP_bilinear_form}.
See~\cite{lee2024low} for a detailed comparison.
\end{remark}


\subsection{A $C^0$ Weak Galerkin Method for Biharmonic Problems}

Therefore, we propose a $C^0$ weak Galerkin ($C^0$-WG) method, formulated using the discrete weak Laplacian and the stabilizer that are compatible with the finite-element functions $v_h\in V_h$.
In the following, we first apply the $C^0$-WG method to the biharmonic problem~\eqref{sys:governing}.
\begin{algorithm}[H]
\caption*{\textbf{$C^0$ Weak Galerkin Method for Biharmonic Problems}}\label{alg: modified}
Find $u_h\in V_h$ such that
\begin{equation}\label{eqn: discrete_weak_biharnomic_problem}
    a_w(u_h,v_h)=(f,v_h),\quad \forall v_h\in V_h,
\end{equation}
where 
\begin{align}
    a_w(v,w) = (\Delta _wv,\Delta_w w)_{\mathcal{T}_h} + \frac{1}{4}\langle h_e^{-1}\jump{\nabla v}\cdot\bne,\jump{\nabla w}\cdot\bne\rangle_{\Eho}.\label{eqn: bilinear_aw}
\end{align}
The weak Laplacian $\Delta_wv|_T\in P_0(T)$ is computed using the explicit formula \eqref{eqn: discreteweakL}, which depends solely on $\nabla v$.
The stabilizer, as defined in \eqref{eqn: final_stabz}, is analogous to the penalty term in the $C^0$-IP bilinear form~\eqref{eqn: C0IP_bilinear_form}, but it does not include a penalty parameter.
\end{algorithm}

Following~\cite{mu2014c,mu2014weak}, we define the mesh-dependent norm associated with the bilinear form $a_w(\cdot,\cdot)$ by
\begin{equation}\label{eqn: defi_w_norm}
    \wnorm{v}^2:= a_w(v,v),\quad\forall v\in V_h.
\end{equation}
Then, the following coercivity and continuity properties hold with respect to this norm.
\begin{lemma}
For any $v,w\in V_h$, the bilinear form $a_w(\cdot,\cdot)$ satisfies
\begin{align}
    &a_w(v,v)=\wnorm{v}^2,\label{eqn: coera}\\
    &|a_w(v,w)|\leq \wnorm{v}\wnorm{w}.\label{eqn: contia}
\end{align}
Therefore, the discrete biharmonic problem \eqref{eqn: discrete_weak_biharnomic_problem} is well-posed.
\end{lemma}

\begin{remark}
    Our proposed method~\eqref{eqn: discrete_weak_biharnomic_problem} achieves the optimal order of convergence (first order with respect to $h = \max_{T \in \mathcal{T}_h} h_T$ in the norm $\norm{\cdot}_w$), following the convergence analysis developed in~\cite{mu2014c,mu2014weak}.  
This convergence order is consistent with the optimal convergence order obtained by the $C^0$-IP method~\cite{brenner2005c}.  
Moreover, the norm equivalence that we will establish in Lemma~\ref{lemma: norm_equiv} further supports this observation.  
Corresponding numerical results will be presented in Section~\ref{subsec: nume_sim_biharnomic}.
\end{remark}

\begin{remark}
    The \(C^0\)-WG method~\eqref{eqn: discrete_weak_biharnomic_problem} does not require a penalty parameter, whereas the $C^0$-IP method~\eqref{eqn: C0IP_biharmonic} relies on a sufficiently large penalty parameter \(\rho\).
    Furthermore, as shown in \eqref{eqn: discreteweakL} and \eqref{eqn: final_stabz}, the bilinear form \(a_w(\cdot, \cdot)\) is computed solely using \(\nabla v|_T\) for all \(T \in \mathcal{T}_h\), in contrast to the $C^0$-IP bilinear form~\eqref{eqn: C0IP_bilinear_form}, which involves second-order derivatives. The proposed bilinear form is assembled using the \(L^2\) inner product of weak Laplacians and a parameter-free penalty term, whereas the $C^0$-IP method includes additional trace terms.
    As a result, the proposed method is penalty-parameter-free and offers reduced computational complexity  (see~\cite{lee2024low} for a detailed comparison).
\end{remark}


\subsection{Discretizing the Optimal Control Problems}\label{subsec: discretize_OCPs}
In this section, we apply the proposed $C^0$-WG method to the reduced optimal control problems~\eqref{reduced_cost} to derive its discrete formulation.
\begin{algorithm}[H]
\caption*{\textbf{Discrete Optimal Control Problems}}\label{alg: discrete_ocp}
Find $\bar{y}_h\in U_h$ such that
\begin{equation}\label{discrete_ocp}
    \bar{y}_h = \argmin_{y_h\in U_h}\frac{1}{2}\left[ \beta a_w(y_h,y_h) + \|y_h-y_d\|^2_{L^2(\Omega;\chi)}\right],
\end{equation}
where
\begin{equation}\label{discrete_admissible}
    U_h=\{y_h\in V_h\,:\, I_h\psi_-\leq I_h y_h\leq I_h \psi_+\},
\end{equation}
and $I_h$ is the $P_1$ nodal interpolation operator associated with $\cT_h$.
\end{algorithm}

Let $\Pi_h: H^2(\Omega) \cap H^1_0(\Omega) \to V_h$ denote the Lagrange nodal interpolation operator.
By \eqref{admissible} and \eqref{discrete_admissible}, $\Pi_h$ maps $U$ into $U_h$, and $U_h$ is nonempty.
Therefore, the discrete problem \eqref{discrete_ocp} has a unique solution $\bar{y}_h \in U_h$, characterized by the following discrete variational inequality:
\begin{equation}\label{discrete_vi_ocp}
    \mathcal{A}_w(\bar{y}_h,z_h-\bar{y}_h)-\iO y_d(z_h-\bar{y}_h)\,d\chi\geq 0,\quad\forall z_h\in U_h,
\end{equation}
where 
\begin{equation}\label{eqn: defi_A_w}
    \mathcal{A}_w(y_h,z_h) = \beta a_w(y_h,z_h) + \int_\Omega y_hz_h\,d\chi.
\end{equation}
This bilinear form serves as the discrete counterpart to the continuous bilinear form
\begin{equation}
    \mathcal{A}(y,z) := \beta \int_\Omega (\Delta y)(\Delta z)\,d\bx + \int_\Omega yz\,d\chi,\label{eqn: defi_A}
\end{equation}
for all $y,\,z\in H^2(\Omega)\cap H^1_0(\Omega)$, which appears in~\eqref{vi} and \eqref{kkt}.

\section{Convergence Analysis}
\label{sec: convergence_analysis}

In this section, we present the convergence analysis of the proposed $C^0$-WG method for the constrained optimal control problems. We begin by recalling the Lagrange nodal interpolation operator $\Pi_h:H^2(\Omega)\cap H^1_0(\Omega)\to V_h$ and its interpolation estimates from~\cite{brenner2024c0}:
\begin{subequations}\label{sys: Pih}
\begin{alignat}{2}
&\norm{\bar{y}-\Pi_h\bar{y}}_{0}\leq Ch^{2+\tau},\label{eqn: Pih_0}\\
&|\bar{y}-\Pi_h\bar{y}|_1\leq Ch^{1+\tau},\label{eqn: Pih_1}\\
&\norm{\bar{y}-\Pi_h\bar{y}}_{L^\infty(\Omega)}\leq Ch^{1+\tau},\label{eqn: Pih_inf}\\
&\norm{\bar{y}-\Pi_h\bar{y}}_{L^2(\Omega;\chi)}\leq Ch^{1+\tau},\label{eqn: Pih_chi}\\
&\norm{\bar{y}-\Pi_h\bar{y}}_h\leq Ch^\tau,\label{eqn: Phi_h}
\end{alignat}
\end{subequations}
where
\begin{equation*}
    \tau = 
    \left\{\begin{array}{ll}
        \alpha, & \text{if $\Th$ is quasi-uniform,} \\
        1-\epsilon, & \text{if $\Th$ is graded around the corners of $\Omega$} \\
        & \quad \text{where the interior angles are greater than $\pi/2$.}
    \end{array}\right.
\end{equation*}
We also define the operator $\Theta_h:H^2(\Omega)\cap H_0^1(\Omega)\rightarrow \mathcal{W}_2(\Th)$ as
\begin{equation}\label{eqn: operTheta_defi}
    \Theta_h\bar{y} = \{\Pi_h\bar{y},\Pi_h\bar{y},Q_{\bn}(\nabla \bar{y}\cdot\bn_e)\bn_e\}\in \mathcal{W}_2(\Th),
\end{equation}
where $Q_\bn$ is the $L^2$-projection operator onto $P_1(e)$ for each $e\in \Eh$, as introduced in~\cite{mu2014c}. This definition satisfies the following commutative property~\cite{mu2014c}:
\begin{equation}\label{eqn: comm_prop}
    \Delta_w\Theta_h \bar{y}= \mathbb{Q}_h\Delta \bar{y},
\end{equation}
where $\mathbb{Q}_h$ denotes the local $L^2$-projection onto $P_0(T)$.

We recall the mesh-dependent norms~\eqref{eqn: defi_h_norm} and \eqref{eqn: defi_w_norm} used in the convergence analysis:
\begin{align*}
    \hnorm{v}^2 &= (D^2v,D^2v)_{\Th} + \langle h_e^{-1}\jump{\nabla v}\cdot\bne,\jump{\nabla v}\cdot\bne\rangle_{\Eho}, \\
    \wnorm{v}^2&= (\Delta_w v,\Delta_w v)_{\Th} + \frac{1}{4}\langle h_e^{-1}\jump{\nabla v}\cdot\bne,\jump{\nabla v}\cdot\bne\rangle_{\Eho}.
\end{align*}
For $v\in H^2(\Omega)\cap H_0^1(\Omega)$, we have $\Delta_w v = \Delta v$ on each $T\in\Th$ by~\eqref{eqn: relation1}, and therefore $\norm{v}_w = \norm{\Delta v}_{0,\Th}$, which implies
\begin{equation}\label{eqn: norm_bd}
    \wnorm{v} \leq C \hnorm{v},
\end{equation}
for some constant \( C > 0 \).
Moreover, in the finite-dimensional space \( V_h \), the following norm equivalence holds.
\begin{lemma}\label{lemma: norm_equiv}
For any $v_h\in V_h$, there exist positive constants $\gamma_*$ and $\gamma^*$, independent of $h: = \max_{T\in \Th} h_T$, such that
\begin{equation}\label{eqn: norm_equiv}
    \gamma_*\hnorm{v_h} \leq \wnorm{v_h}\leq \gamma^*\hnorm{v_h},\quad\forall v_h\in V_h.
\end{equation}
\end{lemma}
\begin{proof}
     For any $v_h\in V_h$, choosing $p_0 = \Delta_w v_h$ in \eqref{eqn: relation1} yields
    \begin{equation*}
        \norm{\Delta_w v_h}_{0,\Th}^2=\left(\Delta v_h,\Delta_w v_h\right)_{\Th}+\langle\jump{\nabla v_h}\cdot \bn_e,\avg{\Delta_w v_h}\rangle_{\Eho}.
    \end{equation*}
    Applying the Cauchy-Schwarz inequality gives
    \begin{equation*}
        \left(\Delta v_h,\Delta_w v_h\right)_{\Th}\leq \norm{\Delta v_h}_{0,\Th}\norm{\Delta_w v_h}_{0,\Th}.
    \end{equation*}
    Using the Cauchy-Schwarz inequality together with the trace inequality~\eqref{eqn: trace}, we estimate the jump term:
    \begin{align*}
        \langle\jump{\nabla v_h}\cdot \bn_e,\avg{\Delta_w v_h}\rangle_{\Eho}&\leq \norm{h_e^{-1/2}\jump{\nabla v_h}\cdot\bne}_{0,\Eho}\norm{h_e^{1/2}\avg{\Delta_w v_h}}_{0,\Eho}\\
        &\leq C \norm{h_e^{-1/2}\jump{\nabla v_h}\cdot\bne}_{0,\Eho}\norm{\Delta_w v_h}_{0,\Th}.
    \end{align*}
    Combining these estimates yields
    \begin{equation*}
        \norm{\Delta_w v_h}_{0,\Th}\leq C\hnorm{v_h}.
    \end{equation*}
    To obtain the reverse bound, we choose \( p_0 = \Delta v_h \) in~\eqref{eqn: relation1} and apply similar arguments to obtain
    \begin{equation*}
        \norm{\Delta v_h}_{0,\Th}\leq C\wnorm{v_h}.
        \hfill \qedhere
    \end{equation*}
\end{proof}

We now introduce a mesh-dependent norm for the optimal control problems, as defined in~\cite{brenner2024c0}:
\begin{equation*}
    \trinorm{z}_{\mathcal{H}}^2:= \beta\hnorm{z}^2 + \norm{z}_{L^2(\Omega;\chi)}^2,
\end{equation*}
and we define its counterpart in the weak Galerkin framework:
\begin{equation*}
    \trinorm{z}_{\mathcal{W}}^2:= \beta\wnorm{z}^2 + \norm{z}_{L^2(\Omega;\chi)}^2.
\end{equation*}
Using this norm, we establish the following properties of the bilinear form $\mathcal{A}_w(\cdot,\cdot)$ in \eqref{eqn: defi_A_w} with respect to $\trinorm{\cdot}_{\mathcal{W}}$. 
\begin{lemma}
For any \( w_h, z_h \in V_h \), the bilinear form satisfies the following properties. There exist constants \( C_\dag > 0 \) and \( C_\ddag > 0 \) such that
\begin{align}
    &\mathcal{A}_w(z_h,z_h)\geq C_\dag\trinorm{z_h}_{\mathcal{W}}^2,\label{eqn: coerA}\\
    &\mathcal{A}_w(w_h,z_h)\leq C_\ddag\trinorm{w_h}_{\mathcal{W}}\trinorm{z_h}_{\mathcal{W}}.\label{eqn: contiA}
\end{align}    
\end{lemma}
\begin{proof}
    These estimates follow from the coercivity and continuity results given in~\eqref{eqn: coera} and~\eqref{eqn: contia}. 
\end{proof}

In addition, let \( W_h \subset H^2(\Omega) \cap H_0^1(\Omega) \) denote the Hsieh-Clough-Tocher finite element space~\cite{ciarlet1974element} associated with \( \Th \). We define an operator \( E_h : V_h \rightarrow W_h \) by vertex averaging~\cite{brenner2010posteriori,brenner2011c}, such that
\begin{equation*}
    (E_hz_h)(\mathbf{x}_v) = z_h(\mathbf{x}_v),
\end{equation*}
for every vertex \( \mathbf{x}_v \) of \( \Th \). The operator \( E_h \) satisfies the following estimates, as established in~\cite{brenner2024c0}: for every \( z_h \in V_h \),
\begin{subequations}\label{sys: Eh}
\begin{alignat}{2}
&\norm{z_h-E_hz_h}_{0,\Th}\leq Ch^{2}\hnorm{z_h},\label{eqn: Eh_0}\\
&|z_h-E_hz_h|_{1,\Th}\leq Ch\hnorm{z_h},\label{eqn: Eh_1}\\
&|E_hz_h|_{2,\Th}\leq C\hnorm{z_h},\label{eqn: Eh_2}\\
&\norm{z_h-E_hz_h}_{L^2(\Omega;\chi)}\leq Ch(1+|\ln{h}|)^{1/2}\hnorm{z_h},\label{eqn: estimate_z_Ez}
\end{alignat}
\end{subequations}
where \( C > 0 \) is a constant that may differ across estimates but depends only on the shape regularity of \( \Th \).


\subsection{Key Estimates}\label{subsec: key_estimates}

To proceed with the convergence analysis, we first establish an estimate for a key term.
\begin{lemma}\label{lemma: estimate_aw_a}
Let \( \bar{y}\in H^2(\Omega)\cap H_0^1(\Omega) \) be the optimal state of the continuous optimal control problems \eqref{reduced_cost}.
There exists a positive constant $C$, independent of $h$, such that:
    \begin{equation}
    a_w(\Pi_h\bar{y},z_h) - a(\bar{y},E_hz_h)\leq Ch^\tau \norm{z_h}_h,\quad \forall z_h\in V_h. \label{eqn: estimate_aw_a}
    \end{equation}
\end{lemma}
\begin{proof}
We decompose the difference into four terms:
\begin{align*}
    a_w(\Pi_h\bar{y},z_h) - a(\bar{y},E_hz_h) &= \underbrace{a_w(\Pi_h\bar{y},z_h) - a_w(\Theta_h\bar{y},z_h)}_{(\mathrm{I})} + \underbrace{a_w(\Theta_h\bar{y},z_h) - a_h(\bar{y},z_h)}_{(\mathrm{II})}\\
    &\quad\quad+\underbrace{a_h(\bar{y},z_h) - a_h(\Pi_h\bar{y},z_h)}_{(\mathrm{III})} + \underbrace{a_h(\Pi_h\bar{y},z_h) - a(\bar{y},E_hz_h)}_{(\mathrm{IV})},
    \end{align*}
and estimate each term individually.

\vskip 5pt
    \noindent\textbf{Estimate of $(\mathrm{I})$}: Expanding the definition of \( a_w(\cdot, \cdot) \) from \eqref{eqn: WG_biharmonic}, we write
        \begin{equation*}
            (\mathrm{I}) = a_w(\Pi_h\bar{y}-\Theta_h\bar{y},z_h) = (\Delta_w(\Pi_h\bar{y}-\Theta_h\bar{y}),\Delta_wz_h)_{\Th}+s(\Pi_h\bar{y}-\Theta_h\bar{y},z_h).
        \end{equation*}
    For the first term, we use the definition of the discrete weak Laplacian \eqref{discreteL}, the projection operator \eqref{eqn: operTheta_defi}, the trace inequality \eqref{eqn: trace}, and the interpolation estimates \eqref{sys: Pih}:
        \begin{align*}
            (\Delta_w(\Pi_h\bar{y}-\Theta_h\bar{y}),\Delta_wz_h)_{\Th} &= \sum_{T\in\Th}\langle\avg{\nabla\Pi_h\bar{y}}\cdot\bn_T-Q_{\bn}(\nabla\bar{y}\cdot\bn_T),\Delta_wz_h\rangle_{\partial T}\\
            &=\sum_{T\in\Th}\langle(\avg{\nabla\Pi_h\bar{y}}-\nabla\bar{y})\cdot\bn_T,\Delta_wz_h\rangle_{\partial T}\\
            &\leq \sum_{T\in\Th}\norm{h_T^{-1/2}\avg{\nabla\Pi_h\bar{y} - \nabla\bar{y}}}_{0,\partial T} \norm{h_T^{1/2}\Delta_wz_h}_{0,\partial T}\\
            &\leq Ch^\tau\norm{\Delta_wz_h}_{0,\Th}.
        \end{align*}
    For the stabilizer defined in~\eqref{penalty} (with $\bm{\nu}_g=\avg{\nabla v_h}$ for $v_h\in V_h$), by applying the regularity of \( \bar{y} \), the trace inequality \eqref{eqn: trace}, and the interpolation estimates \eqref{sys: Pih}, we obtain:
        \begin{align*}
            &s(\Pi_h\bar{y}-\Theta_h\bar{y},z_h)\\
            &\quad\quad=s(\Pi_h\bar{y},z_h) - s(\Theta_h\bar{y},z_h)\\
            &\quad\quad= \frac{1}{4}\langle h_e^{-1} \jump{\nabla\Pi_h\bar{y}}\cdot\bne,\jump{\nabla z_h}\cdot\bne\rangle_{\Eho}- \frac{1}{2} \langle h_e^{-1} \jump{\nabla \Pi_h\bar{y}}\cdot\bn_e,\jump{\nabla z_h}\cdot\bn_e\rangle_{\Eho}\\
            &\quad\quad=-\frac{1}{4} \langle h_e^{-1} \jump{\nabla\Pi_h\bar{y}-\nabla\bar{y}}\cdot\bne,\jump{\nabla z_h}\cdot\bne\rangle_{\Eho}\\
            &\quad\quad\leq C\norm{h_e^{-1/2}\jump{\nabla\Pi_h\bar{y}-\nabla\bar{y}}}_{0,\Eho}\norm{h_e^{-1/2}\jump{\nabla z_h}\cdot\bne}_{0,\Eho}\\
            &\quad\quad\leq Ch^\tau (s(z_h,z_h))^{1/2}.
        \end{align*}
        The second identity holds because $\jump{Q_\bn(\nabla\bar{y}\cdot\bn_e)\bne}=\mathbf{0}$ on $e\in\Eho$.
    Thus, we conclude:
        \begin{equation*}
            (\mathrm{I})\leq Ch^\tau \wnorm{z_h} \leq Ch^\tau\hnorm{z_h},
        \end{equation*}
    using the norm equivalence \eqref{eqn: norm_equiv}.

\vskip 5pt
\noindent\textbf{Estimate of $(\mathrm{II})$:} Using the identity \eqref{eqn: relation1} and the commutative property \eqref{eqn: comm_prop}, we obtain:
\begin{align*}
    (\Delta_w\Theta_h\bar{y},\Delta_wz_h)_{\Th} &= (\Delta_w\Theta_h\bar{y},\Delta z_h)_{\Th}+\langle \avg{\Delta_w\Theta_h\bar{y}} ,\jump{\nabla z_h}\cdot\bn_e\rangle_{\Eho}\\
    &=(\mathbb{Q}_h\Delta\bar{y},\Delta z_h)_{\Th}+\langle \avg{\mathbb{Q}_h\Delta\bar{y}} ,\jump{\nabla z_h}\cdot\bn_e\rangle_{\Eho}\\
    &=(\Delta\bar{y},\Delta z_h)_{\Th}+\langle \avg{\mathbb{Q}_h\Delta\bar{y}} ,\jump{\nabla z_h}\cdot\bn_e\rangle_{\Eho}\\
    &=a_h(\bar{y},z_h) + \langle\avg{\mathbb{Q}_h\Delta\bar{y}-\Delta\bar{y}},\jump{\nabla z_h}\cdot\bn_e\rangle_{\Eho}.
\end{align*}
Using this result, along with the trace inequality~\eqref{eqn: trace} and the estimates \eqref{sys: Pih}, we estimate the term $(\mathrm{II})$:
\begin{align*}
    (\mathrm{II}) &= a_w(\Theta_h\bar{y}, z_h) - a_h(\bar{y},z_h)\\
    & = s(\Theta_h\bar{y},z_h)+\langle\avg{\mathbb{Q}_h\Delta\bar{y}-\Delta\bar{y}},\jump{\nabla z_h}\cdot\bn_e\rangle_{\Eho}\\
    &= \frac{1}{2} \langle h_e^{-1} \jump{\nabla \Pi_h\bar{y}-\nabla\bar{y}}\cdot\bn_e,\jump{\nabla z_h}\cdot\bn_e\rangle_{\Eho}+\langle\avg{\mathbb{Q}_h\Delta\bar{y}-\Delta\bar{y}},\jump{\nabla z_h}\cdot\bn_e\rangle_{\Eho}\\
    &\leq C\left( \norm{h_e^{-1/2}\jump{\nabla \Pi_h\bar{y}-\nabla\bar{y}}}_{0,\Eho}\norm{h_e^{-1/2}\jump{\nabla z_h}\cdot\bne}_{0,\Eho}\right.\\
    &\qquad\qquad\qquad\qquad\qquad  \left.+\norm{h_e^{1/2}\avg{\mathbb{Q}_h\Delta\bar{y}-\Delta\bar{y}}}_{0,\Eho}\norm{h_e^{-1/2}\jump{\nabla z_h}\cdot\bne}_{0,\Eho}\right)\\
    &\leq Ch^\tau\norm{z_h}_h.
\end{align*}
The term involving $\mathbb{Q}_h\Delta\bar{y}-\Delta \bar{y}$ is bounded using the Bramble-Hilbert lemma~\cite{bramble1970estimation,dupont1980polynomial}, combined with the regularity of $\bar{y}$, following the approach presented in~\cite[Lemma 2.5]{brenner2012quadratic}:
\begin{align*}
    \sum_{e\in\Eho}\norm{h_e^{1/2}\avg{\mathbb{Q}_h\Delta\bar{y}-\Delta\bar{y}}}_{0,e}^2&\leq C\sum_{e\in\Eho}\sum_{T\in \mathcal{T}_e}|\mathbb{Q}_h\Delta\bar{y}-\Delta\bar{y}|_{0,T}^2\\
    &\leq Ch^{2\tau}\sum_{e\in\Eho}\sum_{T\in \mathcal{T}_e}|\bar{y}|_{H^{2+\tau}(Q_e)}^2\\
    &\leq Ch^{2\tau}|\bar{y}|_{H^{2+\tau}(\Omega)}^2.
\end{align*}
Here, $\mathcal{T}_e$ denotes the collection of elements sharing $e\in \Eho$, and $Q_e$ represents the union of the two adjacent elements in $\mathcal{T}_e$.

\vskip 5pt
\noindent\textbf{Estimate of $(\mathrm{III})$:} Using the continuity result \eqref{eqn: hcontia} and the interpolation estimate \eqref{eqn: Phi_h}, we obtain
\begin{align*}
    (\mathrm{III}) = a_h(\bar{y}-\Pi_h\bar{y},z_h)\leq C\norm{\bar{y} - \Pi_h\bar{y}}_h\norm{z_h}_h\leq Ch^\tau\norm{z_h}_h.
\end{align*}

\noindent\textbf{Estimate of $(\mathrm{IV})$:} This term quantifies the discrepancy between the bilinear forms \( a_h(\cdot,\cdot) \) and \( a(\cdot,\cdot) \). We invoke the estimate provided in~\cite{brenner2024c0}, which yields:
\begin{equation*}
    (\mathrm{IV}) = a_h(\Pi_h\bar{y},z_h) - a(\bar{y},E_hz_h)\leq Ch^\tau\norm{z_h}_h.
    \hfill
\end{equation*}
Combining the above bounds completes the proof.
\end{proof}

Based on the estimate \eqref{eqn: estimate_aw_a} in Lemma~\ref{lemma: estimate_aw_a}, we establish the following estimate for the bilinear forms associated with the optimal control problems.
\begin{lemma}\label{lemma: estimate_Aw_A}
There exists a positive constant C, independent of $h$, such that
    \begin{equation}
        \mathcal{A}_w(\Pi_h\bar{y},z_h) - \mathcal{A}(\bar{y},E_hz_h)\leq Ch^\tau\trinorm{z_h}_{\mathcal{H}},\quad\forall z_h\in V_h,\label{eqn: estimate_Aw_A}
    \end{equation}
    where the bilinear form $\mathcal{A}(\cdot,\cdot)$ is defined in \eqref{eqn: defi_A}.
\end{lemma}

\begin{proof}
    We begin by expanding the bilinear forms using their definitions:
    \begin{align*}
        \mathcal{A}_w(\Pi_h\bar{y},z_h) - \mathcal{A}(\bar{y},E_hz_h) &= \beta\left[ a_w(\Pi_h\bar{y},z_h) -a(\bar{y},E_hz_h)\right]\\
        & \quad\quad + \int_\Omega(\Pi_h\bar{y}-\bar{y})z_h\,d\chi + \int_\Omega \bar{y}(z_h - E_hz_h)\,d\chi 
    \end{align*}
    The fourth-order term is bounded by \( C h^\tau \trinorm{z_h}_{\mathcal{H}} \) using estimate \eqref{eqn: estimate_aw_a}. For the remaining two terms, we apply the Cauchy-Schwarz inequality and the estimate \eqref{eqn: Pih_chi}:
    \begin{equation*}
        \int_\Omega(\Pi_h\bar{y}-\bar{y})z_h\,d\chi\leq  \norm{\Pi_h\bar{y}-\bar{y}}_{L^2(\Omega,\chi)}\norm{z_h}_{L^2(\Omega,\chi)}\leq Ch^\tau\trinorm{z_h}_\mathcal{H}.
    \end{equation*}
    Since \( \| \bar{y} \|_{L^2(\Omega; \chi)} \leq C_{\bar{y}} < \infty \), by its regularity, and using the estimate \eqref{eqn: estimate_z_Ez} for the operator \( E_h \), we obtain
    \begin{equation*}
        \int_\Omega\bar{y}(z_h-E_hz_h)\,d\chi\leq \norm{\bar{y}}_{L^2(\Omega,\chi)}\norm{z_h-E_hz_h}_{L^2(\Omega,\chi)}\leq Ch^\tau\trinorm{z_h}_\mathcal{H}.
    \end{equation*}
    Combining the above bounds completes the proof.
\end{proof}


\subsection{Convergence Results}\label{subsec: convergence_analysis}

Consequently, we establish the following error estimate, which supports the convergence of our proposed numerical method.
\begin{theorem}\label{thm: aux_error_estimate}
Let \( \bar{y} \) be the optimal state of the continuous optimal control problems \eqref{reduced_cost}, and let \( \bar{y}_h \) be the discrete optimal state of the discrete optimal control problems \eqref{discrete_ocp}. Then, there exists a constant \( C > 0 \), independent of \( h \), such that
    \begin{equation}
        \trinorm{\Pi_h\bar{y}-\bar{y}_h}_{\mathcal{W}}\leq Ch^\tau.\label{eqn: error_estimate_ax}
    \end{equation}
\end{theorem}

\begin{proof}
    From the coercivity \eqref{eqn: coerA} of \( \mathcal{A}_w(\cdot, \cdot) \), the definition of the bilinear form, and the discrete variational inequality~\eqref{discrete_vi_ocp}, we have
    \begin{align*}
        C\trinorm{\Pi_h\bar{y}-\bar{y}_h}_\mathcal{W}^2 &\leq \mathcal{A}_w(\Pi_h\bar{y}-\bar{y}_h,\Pi_h\bar{y} - \bar{y}_h)\\
        &\leq \underbrace{\mathcal{A}_w(\Pi_h\bar{y},\Pi_h\bar{y}-\bar{y}_h)}_{(\mathrm{I})}+\underbrace{\left[-\int_\Omega y_d(\Pi_h\bar{y}-\bar{y}_h)\,d\chi\right]}_{(\mathrm{II})}.
    \end{align*}
We estimate the terms \((\mathrm{I})\) and \((\mathrm{II})\) separately.

\vskip 5pt
    \noindent \textbf{Estimate of \((\mathrm{I})\):}  
    Choosing \( z_h = \Pi_h \bar{y} - \bar{y}_h \) in Lemma~\ref{lemma: estimate_Aw_A} gives
    \begin{align*}
        \mathcal{A}_w(\Pi_h\bar{y},\Pi_h\bar{y}-\bar{y}_h)\leq \mathcal{A}(\bar{y},E_h(\Pi_h\bar{y}-\bar{y}_h)) + Ch^\tau\trinorm{\Pi_h\bar{y}-\bar{y}_h}_\mathcal{H}.
    \end{align*}
  
    \vskip 5pt
    \noindent \textbf{Estimate of \((\mathrm{II})\):}  
    Since \( \| y_d \|_{L^2(\Omega; \chi)} \leq C_{y_d} < \infty \), and using the estimate \eqref{eqn: estimate_z_Ez}, we obtain
    \begin{align*}
        \int_\Omega y_d (E_hz_h - z_h)\,d\chi\leq \norm{y_d}_{L^2(\Omega,\chi)}\norm{E_hz_h-z_h}_{L^2(\Omega,\chi)}\leq Ch^\tau\trinorm{z_h}_\mathcal{H}.
    \end{align*}
    Choosing \( z_h = \Pi_h \bar{y} - \bar{y}_h \), we bound \((\mathrm{II})\) as
    \begin{equation*}
        -\int_\Omega y_d(\Pi_h\bar{y}-\bar{y}_h)\,d\chi\leq -\int_\Omega y_d(E_h(\Pi_h\bar{y}-\bar{y}_h))\,d\chi+Ch^\tau\trinorm{\Pi_h\bar{y}-\bar{y}_h}_\mathcal{H}.
    \end{equation*}

    Combining the two estimates, we obtain
    \begin{align*}
        (\mathrm{I}) + (\mathrm{II}) &\leq Ch^\tau\trinorm{\Pi_h\bar{y} - \bar{y}_h}_\mathcal{H} + \underbrace{\mathcal{A}(\bar{y},E_h(\Pi_h\bar{y}-\bar{y}_h))  -\int_\Omega y_d(E_h(\Pi_h\bar{y}-\bar{y}_h))\,d\chi}_{(\mathrm{III})}.
    \end{align*}
    \vskip 5pt
\noindent \textbf{Estimate of \((\mathrm{III})\):}  
By choosing \( z = E_h(\Pi_h \bar{y} - \bar{y}_h) \) in the KKT conditions~\eqref{kkt} with \eqref{eqn: defi_A}, and applying the estimates from~\cite[Section 4.2]{brenner2024c0}, we obtain:
    \begin{align*}
        \mathrm{(III)}&=\mathcal{A}(\bar{y},E_h(\Pi_h\bar{y}-\bar{y}_h))  -\int_\Omega y_d(E_h(\Pi_h\bar{y}-\bar{y}_h))\,d\chi\\
        &= \int_\Omega E_h (\Pi_h\bar{y} - \bar{y}_h)\,d\xi_- +\int_\Omega E_h (\Pi_h\bar{y} - \bar{y}_h)\,d\xi_+\\
        & \leq Ch^{2\tau} + Ch^\tau\trinorm{\Pi_h\bar{y}-\bar{y}_h}_\mathcal{H},
    \end{align*}
    where $\xi$ is the regular Borel measure defined in~\eqref{lagrange}.

    Finally, applying Young's inequality with a constant \( \alpha > 0 \) such that \( \alpha < 1/C_* \), and using the norm equivalence \eqref{eqn: norm_equiv}, we obtain
    \begin{align*}
        \trinorm{\Pi_h\bar{y} - \bar{y}_h}_{\mathcal{W}}^2 &\leq Ch^{2\tau} + C_*h^\tau\trinorm{\Pi_h\bar{y}-\bar{y}_h}_\mathcal{H}\\
        &\leq Ch^{2\tau} + C_*\left(\frac{h^{2\tau}}{2\alpha} + \frac{\alpha}{2}\trinorm{\Pi_h\bar{y}-\bar{y}_h}_\mathcal{W}^2\right).
    \end{align*}
    Choosing \( \alpha \) sufficiently small so that the last term can be absorbed into the left-hand side yields the desired estimate:
    \begin{equation*}
        \trinorm{\Pi_h\bar{y} - \bar{y}_h}_{\mathcal{W}}\leq Ch^\tau.
        \hfill \qedhere
    \end{equation*}  
\end{proof}

Furthermore, we present the total error estimate, which ensures the convergence of our proposed numerical method.
\begin{theorem}
    Let \( \bar{y} \) be the optimal state of the continuous problems and \( \bar{y}_h \) be the discrete optimal state. Then, there exists a constant \( C > 0 \), independent of \( h \), such that
    \begin{equation}
        \trinorm{\bar{y}-\bar{y}_h}_{\mathcal{W}}\leq Ch^\tau.
    \end{equation}
\end{theorem}

\begin{proof}
   We begin by applying the triangle inequality and the norm boundedness \eqref{eqn: norm_bd}:
        \begin{align*}
        \trinorm{\bar{y} - \bar{y}_h}_{\mathcal{W}}\leq C\trinorm{\bar{y}-\Pi_h\bar{y}}_\mathcal{H} + \trinorm{\Pi_h\bar{y}-\bar{y}_h}_\mathcal{W}
    \end{align*}
    The first term is bounded using \eqref{eqn: Phi_h} and \eqref{eqn: Pih_chi}:
    \begin{equation*}
        \trinorm{\bar{y}-\Pi_h\bar{y}}_\mathcal{H}^2 = \beta\hnorm{\bar{y}-\Pi_h\bar{y}}^2 + \norm{\bar{y}-\Pi_h\bar{y}}^2_{L^2(\Omega,\chi)}\leq Ch^{2\tau}.
    \end{equation*}
    Combining this with the error estimate~\eqref{eqn: error_estimate_ax} in Theorem~\ref{thm: aux_error_estimate}, we conclude that
    \begin{equation*}
        \trinorm{\bar{y} - \bar{y}_h}_{\mathcal{W}}\leq Ch^\tau. 
        \hfill \qedhere
    \end{equation*}
\end{proof}


\section{An Additive Schwarz Preconditioner}\label{sec: preconditioning}

To improve the efficiency of our method, we develop and analyze a one-level additive Schwarz preconditioner for the discrete system in this section, following the framework of~\cite{brenner2022additive}.
We use the notation $\lesssim$ to denote an inequality up to a constant factor. Specifically, for two quantities $A$ and $B$, the expression $A\lesssim B$ means that there exists a constant $C>0$, independent of discretization parameters (e.g., mesh size $h$, overlap size $\delta$, or problem data), such that $A\leq CB$.

We define an operator $\mathsf{A}_h$ on $V_h \to V_h'$ by
$$\langle \mathsf{A}_h u, v \rangle = a_w (u, v), \quad \forall u, v \in V_h.$$
Our goal is to develop a one-level additive Schwarz preconditioner for $\mathsf{A}_h$.
By the coercivity property \eqref{eqn: coera}, we have the following identity:
\begin{equation}\label{eq:normequivAh}
    \langle \mathsf{A}_h v, v \rangle = \wnorm{v}^2, \quad \forall v \in V_h.
\end{equation}

We assume that the domain $\Omega$ is partitioned into overlapping subdomains $\Omega_1, \dots, \Omega_J$ such that $\Omega = \cup_{j=1}^J \Omega_j$, each with diameter $\text{diam}(\Omega_j) \approx H$, and that the boundaries of $\Omega_j$ align with the mesh $\Th$. The overlap among the subdomains is characterized by a parameter $\delta$, and we assume the existence of nonnegative functions $\theta_1, \dots, \theta_J \in C^{\infty} (\bar{\Omega})$ satisfying
\begin{enumerate}
\item[i.] $\theta_j = 0$ on $\Omega \setminus \Omega_j$,
    \item[ii.] $\displaystyle \sum_{j=1}^J \theta_j = 1$ on $\bar{\Omega}$, 
    \item[iii.] $\|\nabla \theta_j \|_{L^{\infty}(\Omega)} \lesssim \delta^{-1}$ and $ \|D^2 \theta_j \|_{L^{\infty}(\Omega)} \lesssim \delta^{-2}$,
\end{enumerate}
where $D^2 \theta_j$ denotes the Hessian of $\theta_j$. Moreover, we assume
\begin{equation}\label{Assump:OneLevelASP}
    \textbf{Overlap assumption: }\text{any point in $\Omega$ belongs to at most $N_c$ subdomains.}
\end{equation}

\begin{remark}
    The construction of $\theta_j$ is standard \cite{Rudin1991Functional}. Given a coarse partition of $\Omega$ consisting of convex elements, such as quadrilaterals in two dimensions or hexaherdra in three dimensions, we define $\Omega_j$, by enlarging each coarse element by an amount of $\delta$, such that each $\Omega_j$ becomes a union of fine mesh elements in $\cT_h$.
\end{remark}
For each subdomain $\Omega_j$, define the local finite element space
$$V_j = \left\{ v \in V_h \,:\, v=0 \text{ on } \Omega \setminus \Omega_j \right\},$$ 
and define the local operator $\mathsf{A}_{j} : V_j \rightarrow V_j'$ by
$$\langle \mathsf{A}_j u_j, v_j \rangle = a_{w,j} (u_j, v_j), \quad \forall u_j, v_j \in V_j,$$
where 
$$    a_{w,j}(u_j,v_j) = (\Delta _w u_j,\Delta_w v_j)_{\mathcal{T}_{h,j}} + \frac{1}{4}\langle h_e^{-1}\jump{\nabla u_j}\cdot\bne,\jump{\nabla v_j}\cdot\bne\rangle_{\mathcal{E}_{h,j}},
$$
with
\begin{equation*}
 \mathcal{T}_{h,j} = \left\{ T \in \Th \,:\, T \subset \Omega_j \right\}\quad\text{and}\quad   \mathcal{E}_{h,j} =\left\{ e \in \Eh \,:\, e \subset \bar{\Omega}_j \setminus \partial \Omega \right\}.
\end{equation*}
We define the local energy norm on $V_j$ by
\begin{equation}\label{eq:normequivAj}
    \norm{v_j}_{w,j}^2 := a_{w,j} (v_j, v_j), \quad \forall v_j \in V_j,
\end{equation}
which induces a norm on $V_j$.

The one-level additive Schwarz preconditioner \( \mathsf{B}_h : V_h' \rightarrow V_h \) is defined by
\[
\mathsf{B}_h = \sum_{j=1}^J \mathsf{I}_j \mathsf{A}_j^{-1} \mathsf{I}_j^\top,
\]
where \( \mathsf{I}_j : V_j \rightarrow V_h \), for \( 1 \leq j \leq J \), denotes the natural injection operator, and \( \mathsf{I}_j^\top : V_h' \rightarrow V_j' \) is its transpose (i.e., the adjoint operator with respect to the duality pairing).
We now turn to estimating the condition number of the preconditioned system \( \mathsf{B}_h \mathsf{A}_h \). The following theorem provides the condition number estimate.

\begin{theorem}
    It holds that 
    \begin{equation}\label{eq:condnumesti}
        \kappa (\mathsf{B}_h \mathsf{A}_h) := \frac{\lambda_{\max} (\mathsf{B}_h \mathsf{A}_h)}{\lambda_{\min} (\mathsf{B}_h \mathsf{A}_h)} \lesssim \delta^{-4},
    \end{equation}
    where $\lambda_{\max}(\mathsf{B}_h \mathsf{A}_h)$ and $\lambda_{\min}(\mathsf{B}_h \mathsf{A}_h)$ denote the largest and smallest eigenvalues of $\mathsf{B}_h \mathsf{A}_h$, respectively.
\end{theorem}
\begin{proof}
    First, let $v \in V_h$ be arbitrary. For any decomposition $v = \sum_{j=1}^J \mathsf{I}_jv_j$ with $v_j\in V_j$, we apply
    \eqref{eq:normequivAh}, \eqref{Assump:OneLevelASP}, Cauchy-Schwarz inequality, and \eqref{eq:normequivAj} to obtain
    $$\langle \mathsf{A}_h v, v \rangle = \norm{v}_w^2 \leq \sum_{j=1}^J \norm{\mathsf{I}_j v_j}_w^2 \lesssim \sum_{j=1}^J \norm{v_j}_{w,j}^2 = \sum_{j=1}^J \langle \mathsf{A}_{j} v_j, v_j \rangle,$$
    which implies
    $$\langle \mathsf{A}_h v, v \rangle \lesssim \min_{\substack{v=\sum_{j=1}^J \mathsf{I}_j v_j \\ v_j \in V_j}} \langle \mathsf{A}_{j} v_j, v_j \rangle .$$
    Therefore, by the standard theory of the additive Schwarz preconditioners \cite{Brenner2008TheMathematical,mathew2008domain,bjorstad1991spectra,toselli2004domain,smith2004domain}, we conclude that
    \begin{equation}\label{eq:PreCondLambdaMax}
        \lambda_{\max} (\mathsf{B}_h \mathsf{A}_h) \lesssim 1.
    \end{equation}
    
    Next, we estimate $\lambda_{\min} (\mathsf{B}_h \mathsf{A}_h)$.
    Let $v\in V_h$ be arbitrary, and define
    \begin{equation}\label{def:vj}
        v_j = \Pi_h (\theta_j v), \quad 1 \leq j \leq J,
    \end{equation}
    where $\Pi_h$ is the Lagrange nodal interpolation operator onto $V_h$.
    Then, by the properties of $\theta_j$, we have $v_j \in V_j$, and 
    $$\sum_{j=1}^J v_j = \sum_{j=1}^J \Pi_h (\theta_j v) = \Pi_h \left( \sum_{j=1}^J \theta_j \right) v = \Pi_h v = v.$$
    The interpolation operator also has the following property in each $T\in\Th$~\cite{brenner2005twoleveladditive}:
    \begin{equation}\label{eqn: standard_interpolation_estimates}
         |\Pi_h (\theta_j v)|_{2,T}\lesssim|\theta_j v|_{2,T}.
    \end{equation}
    Then, using the same argument in the proof of Lemma~\ref{lemma: norm_equiv}, we have
    \begin{equation}\label{eq:Ajvj}
        \langle \mathsf{A}_j v_j, v_j \rangle = \norm{v_j}_w^2 \lesssim 
        \sum_{T\in \mathcal{T}_{h,j}} \snorm{v_j}_{2,T}^2 + \sum_{e\in \cE_{h,j}}  \norm{ h_e^{-1/2}\jump{\nabla v_j}\cdot\bne }_{0,e}^2.
    \end{equation}
    For any element $T\in \mathcal{T}_{h,j}$ in \eqref{eq:Ajvj}, by the definition \eqref{def:vj}, the interpolation property \eqref{eqn: standard_interpolation_estimates}, and the bounds on $\theta_j$, we obtain
    \begin{eqnarray*}
        |v_j|_{2,T}^2 &=& |\Pi_h (\theta_j v)|_{2,T}^2\\
        &\lesssim& |\theta_j v|_{2,T}^2 \\
        &\lesssim& \|\theta_j\|_{L^\infty(T)}^2 |v|_{2,T}^2 + \|\nabla \theta_j \|_{L^\infty(T)}^2 |v|_{1,T}^2 + \|D^2 \theta_j \|_{L^\infty(T)}^2 \|v\|_{0,T}^2 \\
        &\lesssim& |v|_{2,T}^2 + \frac{1}{\delta^2} |v|_{1,T}^2 + \frac{1}{\delta^4} \|v\|_{0,T}^2.
    \end{eqnarray*}
    For any interface $e \in \cE_{h,j}$ in~\eqref{eq:Ajvj}, we estimate using the definition of $v_j$~\eqref{def:vj}, the trace inequality~\eqref{eqn: trace}, the interpolation estimates~\eqref{sys: Pih}, and the bounds on $\theta_j$,

    \begin{align*}
        \| h_e^{-1/2}\jump{\nabla v_j}\cdot\bne \|_{0,e}^2 &=  \| h_e^{-1/2} \jump{\nabla (\Pi_h (\theta_j v))}\cdot\bne \|_{0,e}^2\\
        &\lesssim   \|h_e^{-1/2} \jump{ \nabla (\Pi_h (\theta_j v) - \theta_j v ) }\cdot\bne \|_{0,e}^2 +  \| h_e^{-1/2}\jump{\nabla (\theta_j v)}\cdot\bne \|_{0,e}^2\\
        &\lesssim \sum_{T\in \cT_e} \left[ h_T^{-2} |\Pi_h (\theta_j v) - \theta_j v|_{1,T}^2 + |\Pi_h (\theta_j v) - \theta_j v|_{2,T}^2 \right] +  \| h_e^{-1/2}\jump{\nabla v}\cdot\bne \|_{0,e}^2\\
        &\lesssim \sum_{T\in \cT_e} |\theta_j v|_{2,T}^2 + \| h_e^{-1/2}\jump{\nabla v}\cdot\bne \|_{0,e}^2\\
        &\lesssim \sum_{T\in \cT_e} \left[ |v|_{2,T}^2 + \frac{1}{\delta^2} |v|_{1,T}^2 + \frac{1}{\delta^4} \|v\|_{0,T}^2 \right] + \| h_e^{-1/2}\jump{\nabla v}\cdot\bne \|_{0,e}^2 ,
    \end{align*}
    where $\cT_e$ denotes the set of elements in $\cT_h$ sharing the interface $e$.
    Combining the above results for \eqref{eq:Ajvj},
    summing over $j$, and applying the overlap assumption \eqref{Assump:OneLevelASP} along with the Poincar\'{e}-Friedrichs inequalities~\cite{brenner2004poincare} and the norm equivalence~\eqref{eqn: norm_equiv}, we obtain
    \begin{align*}
        \sum_{j=1}^J \langle  \mathsf{A}_j v_j, v_j \rangle &\lesssim \sum_{T\in \cT_h} \left[ |v|_{2,T}^2 + \frac{1}{\delta^2} |v|_{1,T}^2 + \frac{1}{\delta^4} \|v\|_{0,T}^2 \right] + \| h_e^{-1/2}\jump{\nabla v}\cdot\bne \|_{0,\Eho}^2\\
        &\lesssim \frac{1}{\delta^4} \langle \mathsf{A}_h v, v \rangle.
    \end{align*}
    Hence, the standard theory of additive Schwarz preconditioners yields
    \begin{equation}\label{eq:PreCondLambdaMin}
        \lambda_{\min} (\mathsf{B}_h \mathsf{A}_h) \gtrsim \delta^4.
    \end{equation}
    Therefore, combining \eqref{eq:PreCondLambdaMax} and \eqref{eq:PreCondLambdaMin}, the condition number estimate \eqref{eq:condnumesti} follows.
\end{proof}

\begin{remark}
    Using arguments similar to those in \cite{brenner2005twoleveladditive}, if the subdomains $\Omega_j, 1\leq j\leq J$, are shape regular, then the estimate \eqref{eq:condnumesti} can be improved to
    \begin{equation}
        \kappa (\mathsf{B}_h \mathsf{A}_h) \lesssim H^{-1} \delta^{-3}.
    \end{equation}
\end{remark}


 \section{Numerical Experiments}\label{sec: numerical_experiment}
 
 This section presents a series of numerical experiments to demonstrate the accuracy, efficiency, and robustness of the proposed $C^0$-WG method and the additive Schwarz preconditioner. We begin by applying the method to a biharmonic problem and a constrained optimal control problem that involves general tracking terms and pointwise state constraints. To evaluate solver performance, we compare the convergence behavior of the $C^0$-WG method with that of the $C^0$-IP method in various norms and examine the condition numbers of the resulting linear systems, both with and without preconditioning.

 
 \subsection{Biharmonic Problem}\label{subsec: nume_sim_biharnomic}
Let the computational domain be $\Omega=(-0.5,0.5)^2$.
We consider the biharmonic problem~\eqref{sys:governing} with homogeneous boundary conditions:
\begin{subequations}\label{sys:model}
\begin{alignat*}{2}
\Delta^2 u & = f\quad\text{ in }\Omega,\\
u = \Delta u &=0\quad\text{ on } \partial\Omega.
\end{alignat*}
\end{subequations}
The exact solution is given by
\begin{equation}
u = 10\sin\left(\pi(x+0.5)\right)\sin\left(\pi(y+0.5)\right).
\end{equation}

Table~\ref{table: bih_wL} shows the error measured in the mesh-dependent norm $\norm{u_I-u_h}_w$,
where $u_h$ denotes the numerical solution and 
$u_I$ is the $P_2$ Lagrange interpolant of the exact solution $u$.
Our theoretical results show that the $C^0$-WG method is expected to achieve first-order convergence in the mesh-dependent norm. The table demonstrates that the $C^0$-WG method attains a first-order or higher convergence order, where the observed superconvergence may be attributed to the use of a uniform mesh or the smoothness of the exact solution.
This observation is consistent with the fact that the $C^0$-IP method also exhibits first-order convergence.
\begin{table}[h!]
\center
\footnotesize
\begin{tabular}[h!]{|c||c|c|c|c|c|c|}
\hline 
$h$ & $2^{-3}$& $2^{-4}$& $2^{-5}$& $2^{-6}$& $2^{-7}$& $2^{-8}$\\
\hline \hline
$\norm{u_I - u_h}_{w}$ & 1.16e+1 &3.78e+0&1.27e+0&4.38e-1&1.53e-1&5.36e-2\\
\hline 
Order&-&1.62&1.57&1.54&1.52&1.51\\
\hline
\end{tabular}
\caption{Mesh refinement study for the $C^0$-WG method based on the error between $u_I$ and $u_h$.}
\label{table: bih_wL}
\end{table}


\subsection{Constrained Optimal Control Problem with General Tracking}

Let $\Omega =[-4,4]\times[-4,4]$ and $\beta = 1$.
We consider the optimal control problem introduced in~\cite{brenner2024c0}:
\begin{equation}\label{eg: ocp}
\min_{y\in K}\frac1{2}\left[\|y-y_d\|_{0}^2
+\sum_{j=1}^4w_0(p_j)[y(p_j)-y_0(p_j)]^2+\beta\|\Delta y\|^2_{0}\right],
\end{equation}
where the points are given by
\[
p_1 = (-2.5, \;-2.5),\quad p_2=(2.5,\;-2.5),\quad p_3 = (2.5,\;2.5),\quad p_4=(-2.5,\;2.5),
\]
with weights $w_0(p_j)=100$ for $j=1,\dots,4$, and the admissible set is
\[
K=\{ y\in H^2(\Omega)\cap H^1_0(\Omega)\,:\,y\leq \psi\text{ in }\Omega\}.
\]
The pointwise state constraint is given by $\psi(\bx) = |\bx|^2-1$.
A construction of the exact solution can be found in~\cite{brenner2024c0}.

To see the convergence order in various norms, we compute the relative error $e$ defined as:
\begin{equation*}
     e_\Diamond = \frac{\norm{v-v_h}_\Diamond}{\norm{v}_\Diamond},
\end{equation*}
where $\Diamond=L^\infty(\Omega)$, $L^2(\Omega)$, $H^1(\Omega)$, $w$, or $h$.
We compare numerical results obtained using the $C^0$-WG and $C^0$-IP methods.
As shown in Table~\ref{table: ocp_wL}, the $C^0$-WG method consistently produces more accurate approximations than the $C^0$-IP method across all tested norms.
In particular, the $C^0$-WG method achieves higher convergence orders in the $L^\infty$, $L^2$, and $H^1$ norms.
These results highlight the $C^0$-WG method as a more accurate and computationally efficient alternative to the $C^0$-IP method, especially for constrained optimal control problems with general tracking.
\begin{table}[h!]
\centering
\footnotesize
\begin{tabular}[h!]{|c||c|c||c|c||c|c||c|c|}
\hline
 &  \multicolumn{8}{c|}{$C^0$-WG method}\\
\cline{2-9}
$h$ & $e_{L^\infty(\Omega)}$ & Order & $e_{L^2(\Omega)}$ & Order&
$e_{H^1(\Omega)}$ & Order & $e_w$ & Order \\
\hline
$2^{-1}$&4.84e-1& -    &3.74e-1&  -  &7.30e-1&-	&5.54e-1& -\\
$2^{-2}$&1.62e-1&1.58&3.32e-2&3.50&1.94e-1&1.02&3.02e-1&0.87\\
$2^{-3}$&4.91e-2&1.72&6.95e-3&2.25&1.07e-1&0.86&3.69e-1&-0.29\\
$2^{-4}$&8.24e-3&2.58&2.37e-3&1.55&2.04e-2&2.39&1.06e-1&1.81\\
$2^{-5}$&3.32e-3&1.31&7.40e-4&1.68&5.82e-3&1.80&2.93e-2&1.85\\
$2^{-6}$&9.07e-4&1.87&1.86e-4&1.99&1.54e-3&1.92&7.77e-3&1.91\\
\hline
\hline
 &  \multicolumn{8}{c|}{$C^0$-IP method}\\
\cline{2-9}
$h$ & $e_{L^\infty(\Omega)}$ & Order & $e_{L^2(\Omega)}$ & Order&
$e_{H^1(\Omega)}$ & Order & $e_h$ & Order \\
\hline
$2^{-1}$&3.46e-1& -   &2.07e-1&  -  &5.45e-1&-	&9.64e-1&-\\
$2^{-2}$&1.46e-1&1.25&8.57e-2&1.27&2.55e-1&1.09&6.15e-1&0.65\\
$2^{-3}$&1.41e-1&0.05&4.65e-2&0.88&1.67e-1&0.61&5.11e-1&0.27\\
$2^{-4}$&5.73e-2&1.30&1.78e-2&1.38&7.18e-2&1.22&2.68e-1&0.93\\
$2^{-5}$&2.16e-2&1.41&5.86e-3&1.60&2.81e-2&1.35&1.35e-1&0.99\\
$2^{-6}$&6.71e-3&1.68&1.70e-3&1.79&9.13e-3&1.62&6.40e-2&1.08\\
\hline
\end{tabular}
\caption{Comparison of errors between $\bar{y}$ and $\bar{y}_h$ in various norms for the $C^0$-WG and $C^0$-IP methods ($\rho=100$ for $C^0$-IP \cite{brenner2024c0}).}
\label{table: ocp_wL}
\end{table}

Figure~\ref{fig: optimum_wL} displays the numerical solutions of the optimal state and optimal control for mesh size $h=2^{-6}$.
The optimal state is computed by solving the discrete optimal control problem corresponding to~\eqref{eg: ocp} using $P_2$ Lagrange finite elements. The optimal control is then obtained by applying a discrete Laplacian operator, defined piecewise, as $\bar{u}_h=-\Delta_h \bar{y}_h$.
\begin{figure}[h!]
      \centering
      \begin{subfigure}[b]{0.4\textwidth}
          \centering
          \includegraphics[width=\textwidth]{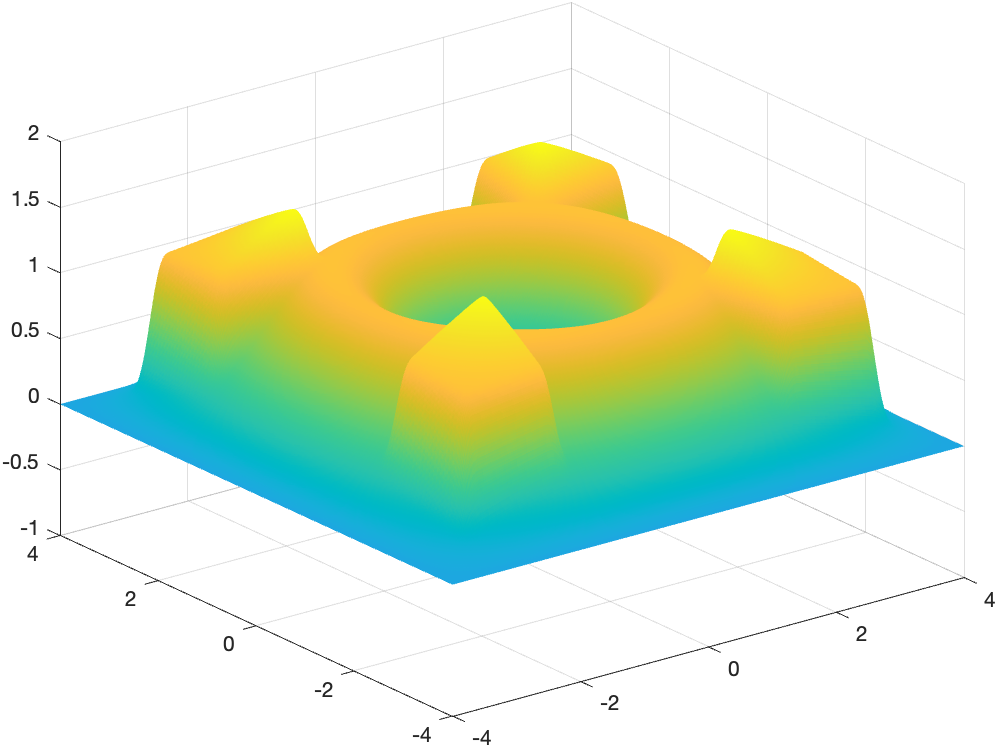}
          \caption{Optimal state}
      \end{subfigure}
      \hspace{10pt}
      \begin{subfigure}[b]{0.4\textwidth}
          \centering
          \includegraphics[width=\textwidth]{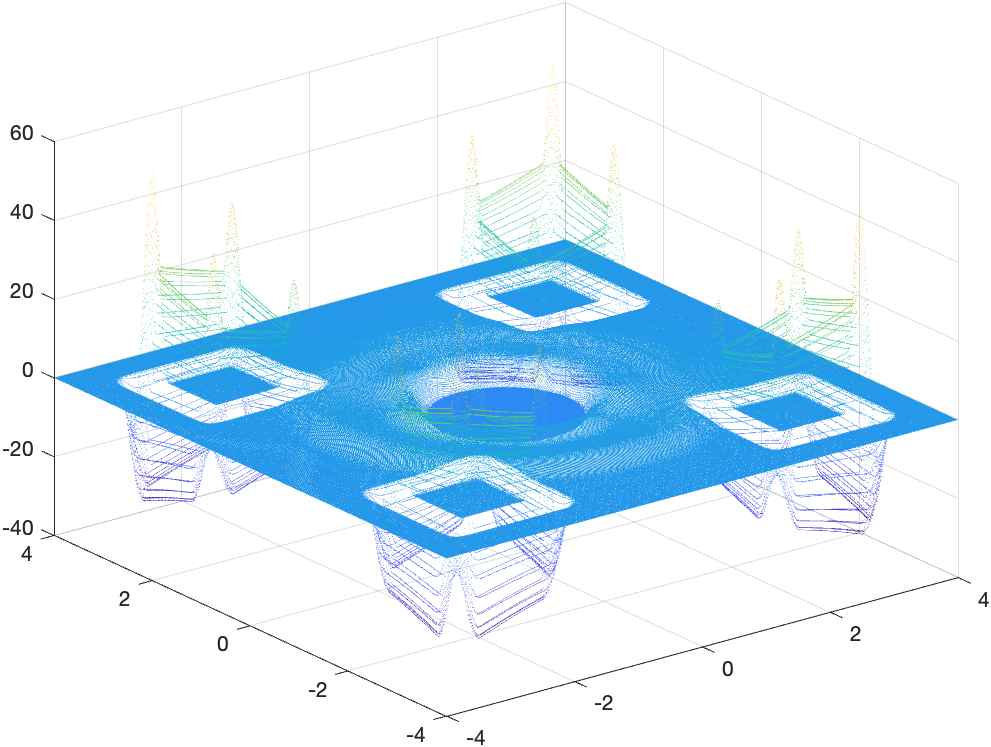}
          \caption{Optimal control}
      \end{subfigure}
      \caption{Quadratic state and piecewise-constant control approximation using the $C^0$-WG method.}
      \label{fig: optimum_wL}
\end{figure}


\subsection{Preconditioning}

We compare the condition numbers of the system matrix arising from the bilinear form~\eqref{eqn: bilinear_aw}, both with and without the additive Schwarz preconditioner (\texttt{ASP}) introduced in Section~\ref{sec: preconditioning}.
We first examine the condition numbers of the unpreconditioned system matrix, denoted by $\kappa(\mathsf{A}_h)$.
As shown in Table~\ref{table:cond_no_preconditioner}, the condition number grows at an order of $\mathcal{O}(h^{-4})$ as the mesh is refined.
The table also reports results for the preconditioned system using \texttt{ASP} with four subdomains placed at the corners of the domain.
With preconditioning, the condition number is reduced by more than two orders of magnitude compared to the unpreconditioned case.
\begin{table}[h!]
    \centering
    \footnotesize
    \begin{tabular}{|c||c|c||c|c||c|c|}
    \hline
        &  \multicolumn{2}{c||}{Without \texttt{ASP}
        } &  \multicolumn{2}{c||}{\texttt{ASP} ($\delta=h$)}
         & 
         \multicolumn{2}{c|}{\texttt{ASP} ($\delta=2h$)} \\
    \cline{2-7}   
       $h$  & $\kappa(\mathsf{A}_h)$ & {Order} & $\kappa(\mathsf{B}_h\mathsf{A}_h)$ & {Order} &
       $\kappa(\mathsf{B}_h\mathsf{A}_h)$ & {Order} \\ 
       \hline
       $2^{-2}$ & 5.61e+2 & - & 5.08e+0 & - &  - & - \\
       $2^{-3}$  & 9.77e+3 & 4.12 & 2.48e+1 & 2.29 & 9.62e+0 & -  \\
       $2^{-4}$   & 1.65e+5 & 4.08 & 1.40e+2 & 2.49 & 4.26e+1 & 2.15  \\
       $2^{-5}$  & 2.68e+6 & 4.02 & 9.87e+2 & 2.81 & 2.63e+2 & 2.63  \\
       $2^{-6}$  & 4.30e+7 & 3.91 & 7.52e+3 & 2.93 & 1.93e+3 & 2.88  \\
       \hline
    \end{tabular}
    \caption{Condition numbers of the linear system with and without \texttt{ASP}.}
    \label{table:cond_no_preconditioner}
\end{table}
Results are shown for two overlap widths, $\delta = h$ and $\delta = 2h$.
In both settings, the growth of the condition number is significantly mitigated, remaining below $\mathcal{O}(h^{-3})$, which indicates improved numerical stability.
Moreover, the condition number is consistently smaller for the larger overlap $\delta = 2h$, highlighting the effectiveness of increased overlap in domain decomposition methods.


\section{Conclusion}
\label{sec:conclusion}

In this paper, we developed a $C^0$ weak Galerkin ($C^0$-WG) method combined with an additive Schwarz preconditioner to efficiently and accurately solve optimal control problems governed by elliptic partial differential equations, featuring general tracking cost functionals and pointwise state constraints.
By eliminating the control variable via the PDE constraint, we reformulated the OCP as a fourth-order variational inequality that characterizes the optimal solution. 
These problems pose significant numerical challenges due to reduced regularity and the need to address ill-conditioned sparse systems arising from discretization.

Our first contribution introduced a weak Galerkin (WG) discretization based on globally continuous quadratic Lagrange elements, eliminating the need for penalty parameters and trace terms required in the symmetric $C^0$ interior penalty method. This WG formulation enables efficient stiffness matrix assembly and robust performance via a parameter-free stabilization.
Our second contribution addressed the ill-conditioning of the resulting linear systems by designing a one-level additive Schwarz preconditioner. This preconditioner significantly improves the condition number of the system, enabling efficient and scalable iterative solvers. Numerical experiments confirmed the accuracy and effectiveness of the proposed $C^0$-WG method and the associated preconditioner for both biharmonic problems and constrained OCPs.

Overall, this work presents a robust and efficient computational framework for high-order constrained OCPs and lays a foundation for future extensions to more complex problems, including time-dependent and nonlinear settings.

	\bibliographystyle{plain}
\bibliography{efficient_OCP}
	
\end{document}